\documentclass[10pt]{amsart} 

\usepackage[utf8]{inputenc}
\usepackage[T1]{fontenc}	
\usepackage[french,english]{babel}	

\usepackage{lmodern}			

\usepackage{graphicx}	

\usepackage[a4paper,plainpages=false,colorlinks,linkcolor=bleuFonce,citecolor=rougeFonce,urlcolor=vertFonce,breaklinks]{hyperref}

\usepackage{placeins}
\usepackage{float} 

\usepackage{color}
\definecolor{vertFonce}{rgb}{0,0.5,0}
\definecolor{numLignes}{rgb}{0.17,0.57,0.7}	
\definecolor{gris}{rgb}{0.5,0.5,0.5}
\definecolor{grisFonce}{rgb}{0.2,0.2,0.2}
\definecolor{orange}{rgb}{1,0.65,0.31}		
\definecolor{orangeFonce}{rgb}{1,0.4,0}
\definecolor{bleuFonce}{rgb}{0,0,0.4}
\definecolor{rougeFonce}{rgb}{0.3,0,0}
\definecolor{rougeWord}{rgb}{0.5,0,0}
\definecolor{vertClair}{rgb}{0.8,1,0.8}
\definecolor{rougeClair}{rgb}{1,0.5,0.5}

\usepackage{pict2e}
\usepackage{multido}

\usepackage{amsfonts,amssymb,amsthm,amsmath} 
\usepackage{dsfont}				
\usepackage{mathrsfs}
\usepackage{yfonts}
 
\newtheorem{lem}{Lemma}[section]
\newtheorem{theorem}{Theorem}
\newtheorem{cor}{Corollary}[section]
\newtheorem{prop}{Proposition}[section]

\newtheorem{remark}{Remark}[section]

\newenvironment{demo}[1][]{%
	\begin{proof}[\textbf{Proof#1}]
	}{%
	\end{proof}
}

\newenvironment{thm}[1][]{%
	\color{rougeWord}\begin{theorem}[#1]
	}{%
	\end{theorem}
}

\newcommand{\step}[1]	{\paragraph{\itshape\bfseries Step #1.}}

\usepackage{stmaryrd}
\newcommand		{\subsetArrow}	{\mathrel{\ooalign{$\subset$\cr%
\hidewidth\raise-.087ex\hbox{$_\shortrightarrow\mkern-1.5mu$}\cr}}}
\newcommand		{\subsetarrow}	{\mathrel{\ooalign{$\subset$\cr%
\hidewidth\raise-1.45ex\hbox{$\vec{}\mkern6mu$}\cr}}}

\newcommand		{\N}			{\mathbb N}			
\newcommand		{\R}			{\mathbb R}			
\renewcommand	{\L}			{\mathcal L}		
\newcommand		{\B}			{\mathscr B}		
\renewcommand	{\P}			{\mathcal P}		
\newcommand		{\PP}			{\mathscr P}		
\newcommand		{\C}			{\mathcal C}		%

\newcommand		{\fb}			{\mathfrak b}
\newcommand		{\cZ}			{\mathcal Z}

\newcommand		{\lt}			{\left}				%
\newcommand		{\rt}			{\right}			%
\renewcommand	{\(}			{\lt(}
\renewcommand	{\)}			{\rt)}
\newcommand		{\lal}			{\langle}			%
\newcommand		{\ral}			{\rangle}			%
\newcommand		{\bangle}[1]	{\lt\lal #1\rt\ral}
\newcommand		{\weight}[1]	{\lt\lal #1\rt\ral}	

\newcommand		{\com}[1]		{\lt[{#1}\rt]}		
\newcommand		{\il}			{[\![}				
\newcommand		{\ir}			{]\!]}				
\newcommand		{\Int}[1]		{\il #1 \ir}
\newcommand		{\n}[1]			{\lt|{#1}\rt|}

\newcommand		{\Nrm}[2]		{\lt\|{#1}\rt\|_{#2}}

\newcommand		{\ii}			{\mathrm{i}}	
\newcommand		{\init}			{\mathrm{in}}
\newcommand		{\loc}			{\mathrm{loc}}

\renewcommand	{\d}			{\mathrm{d}}	
\newcommand		{\dd}			{\,\d}			
\newcommand		{\ddt}[1]		{\frac{\d #1}{\d t}}	

\DeclareMathOperator{\sign}		{sign}
\DeclareMathOperator*{\supess}	{sup\,ess}
\DeclareMathOperator{\diag}		{diag}
\DeclareMathOperator{\tr}		{Tr}			

\newcommand		{\Tr}[1]		{\tr\!\( #1 \)} 	
\newcommand		{\Diag}[1]		{\diag\!\( #1 \)}

\newcommand		{\intd}			{\int_{\R^d}}
\newcommand		{\intdd}		{\int_{\R^{2d}}}
\newcommand		{\iintd}		{\iint_{\R^{2d}}}
\newcommand		{\sumj}			{\sum_{j\in J}}
\newcommand		{\eps}			{\varepsilon}
\newcommand		{\Eps}			{\mathcal{E}}

\usepackage{upgreek}
\renewcommand	{\r}		{\varrho}			
\newcommand		{\op}		{\boldsymbol{\rho}}	
\newcommand		{\opz}		{\mathbf{z}}	
\newcommand		{\gam}		{\boldsymbol{\gamma}}

\newcommand		{\bra}[1]	{\langle #1 |}
\newcommand		{\ket}[1]	{| #1 \rangle}
\newcommand		{\Wh}		{W_{2,\hbar}}		

\newcommand		{\opp}		{\boldsymbol{p}}

\newcommand		{\ch}		{\mathbf{c}_\hbar}
\newcommand		{\Eh}		{\mathcal{E}_\hbar}
\newcommand		{\fh}		{f_\hbar}
\newcommand		{\fht}		{\tilde{\fh}}

\newcommand		{\T}		{\mathsf T}
\newcommand		{\Tx}		{\tilde{\T}}
\newcommand		{\sfS}		{\mathsf{S}}
\newcommand		{\Sx}		{\tilde{\sfS}}

\title[Global Limit from Hartree to Vlasov Equation]{Global Semiclassical Limit from Hartree to Vlasov Equation for Concentrated Initial Data}
\author{Laurent Lafleche$^{1,2}$}
\thanks{$^1$CEREMADE, UMR CNRS 7534, Université Paris-Dauphine, PSL Research University, Place du Maréchal de Lattre de Tassigny, 75775 Paris cedex 16 France, {\tt lafleche@ceremade.dauphine.fr}}
\thanks{$^2$CMLS, \'Ecole polytechnique, CNRS, Universit\'e Paris-Saclay, 91128 Palaiseau cedex, France}
\date{\today}

\begin{document}

\begin{abstract}We prove a quantitative and \textbf{global in time} semiclassical limit from the Hartree to the Vlasov equation in the case of a singular interaction potential in dimension $d\geq 3$, including the case of a Coulomb singularity in dimension $d=3$. This result holds for initial data concentrated enough in the sense that some space moments are initially sufficiently small. As an intermediate result, we also obtain quantitative bounds on the space and velocity moments of even order and the asymptotic behavior of the spatial density due to dispersion effects, uniform in the Planck constant $\hbar$.
\end{abstract}

\maketitle

\bigskip

\renewcommand{\contentsname}{\centerline{Table of Contents}}
\setcounter{tocdepth}{2}	
\tableofcontents


\bigskip
\section{Introduction}
\label{sec:intro}

	The equation governing the dynamics of a large number of interacting particles of density $f = f(t,x,\xi)$ in the phase space is the Vlasov equation
	\begin{equation}\label{eq:Vlasov}\tag{Vlasov}
		\partial_t f +\xi\cdot\nabla_x f + E\cdot\nabla_\xi f = 0,
	\end{equation}
	where $E = -\nabla V$ is the force field corresponding to the mean field potential $V= V(x)$ given by
	\begin{equation*}
		V = K * \rho_f = \intd K(x-y)\,\rho_f(y)\dd y,
	\end{equation*}
	where we denote by $\rho_f(x) := \intd f(x,v)\dd v$ the spatial density and by $K(x)$ the pair interaction potential between two particles at distance $\n{x}$.
	
	The counterpart of the \ref{eq:Vlasov} equation in quantum mechanics is the Hartree equation
	\begin{equation}\label{eq:Hartree}\tag{Hartree}
		i\hbar\,\partial_t \op = [H,\op],
	\end{equation}
	where $\op$ is a self-adjoint Hilbert-Schmidt operator called the density operator and the Hamiltonian is defined by
	\begin{equation*}
		H = -\frac{h^2}{2} \Delta + V.
	\end{equation*}
	In this formula, the potential is defined by $V = K*\rho$ where the spatial density $\rho$ is defined as the diagonal of the kernel $\r(x,y)$ of the operator $\op$, i.e. $\rho(x) = \r(x,x)$.
	
	In this paper, we study in a quantitative way the limit when $\hbar\to 0$ of the \ref{eq:Hartree} equation which is known to converge to the \ref{eq:Vlasov}~equation. The question of the derivation of this equation from the quantum mechanics is a very active topic of research. Non-constructive results in weak topologies have indeed already been proved, including the case of Coulomb interactions, starting from the work of Lions and Paul \cite{lions_sur_1993} and Markowich and Mauser \cite{markowich_classical_1993}. See also \cite{gerard_homogenization_1997, gasser_semiclassical_1998, graffi_mean-field_2003, ambrosio_passage_2010, ambrosio_semiclassical_2011}.
	
	Some more precise quantitative results have also more recently been proved for smooth forces which are always at least Lipschitz in \cite{athanassoulis_strong_2011, amour_classical_2013, amour_semiclassical_2013, benedikter_hartree_2016, golse_mean_2016}. In \cite{golse_schrodinger_2017}, Golse and Paul introduce a pseudo-distance on the model of the Wasserstein-(Monge-Kantorovitch) between classical phase space densities and quantum density operators to get a rate of convergence for the semiclassical limit for Lipschitz forces. This strategy has been used in the recent paper~\cite{lafleche_propagation_2019} of the present author to extend this result to more singular interactions, but only up to a fixed time in the case of potentials with a strong singularity such as the Coulomb interaction.
	
	We also mention the work of Porta et al \cite{porta_mean_2017} and Saffirio \cite{saffirio_mean-field_2018} about the question of the mean-field limit for the Schrödinger equation to the \ref{eq:Hartree} equation for Fermions since this limit is coupled with a semiclassical limit. Results are proved for the Coulomb interaction under assumptions of propagation of regularity along the Hartree dynamics which is still an open problem. Other results about the mean-field limit can be found in \cite{bardos_weak_2000, erdos_derivation_2001, bardos_derivation_2002} where non-quantitative results are established for the Coulomb potential, and more precise limits can be found in \cite{rodnianski_quantum_2009, pickl_simple_2011, golse_mean_2016, mitrouskas_bogoliubov_2019, golse_schrodinger_2017, golse_empirical_2019, golse_derivation_2018} for Bosons and in \cite{frohlich_mean-field_2009, frohlich_microscopic_2011, benedikter_mean-field_2014, benedikter_mean-field_2016, bach_kinetic_2016, petrat_new_2016, porta_mean_2017, petrat_hartree_2017} for Fermions.
	
	Here, we extends the results of \cite{lafleche_propagation_2019} by proving a global in time semiclassical limit under a smallness condition of space moments. We first prove a global in time bound on some modified space moments, from which we obtain the propagation of space and velocity moments. The same kind of results were already known for $\hbar=1$ (see Remark~\ref{rem:existence}), and the main novelty is the fact that the bounds we prove are uniform in $\hbar$. The bound on the velocity is then sufficient to use the theory already used in the above mentioned paper to get a global $L^\infty$ bound on the spatial density and the quantitative semiclassical limit in the quantum Wasserstein pseudo-distance.
	
	The fact that the time decay due to the dispersion properties gives global estimates for the \ref{eq:Vlasov}~equation was already used in \cite{bardos_global_1985}. The modified space moments of order $2$ are linked to a Lyapunov functional reminiscent of the conservation of energy, see \cite{perthame_time_1996, dolbeault_time-dependent_2001}. The propagation of modified space moments was investigated further in \cite{castella_propagation_1999, pallard_moment_2012, pallard_space_2014}.

\subsection{Main results}

	We adopt the same notational conventions as in \cite{lafleche_propagation_2019}. In particular, $L^{p,\infty}$ denotes the weak Lebesgue spaces of functions on $\R^d$ and we define the quantum version of the phase space Lebesgue and weighted Lebesgue spaces as
	\begin{align*}
		\L^p &:= \lt\{\op\in\B(L^2), \Nrm{\op}{\L^p} := h^{-d/p'} \tr(|\op|^p)^\frac{1}{p} < \infty \rt\}
		\\
		\L^p_+ &:= \lt\{\op\in\L^p, \op = \op^* \geq 0\rt\}
		\\
		\L^p(m) &:= \lt\{\op\in\L^p, \op\,m \in \L^p\rt\},
	\end{align*}
	where $\B(L^2)$ denotes the set of bounded linear operators on $L^2$ and $\tr$ denotes the trace. We also define the quantum probability measures by
	\begin{align*}
		\PP &:= \lt\{\op\in\L^1_+, \tr(\op) = 1 \rt\}.
	\end{align*}
	Moreover, in order to ensure well-posedness

	We will denote by $\opp := -i\hbar\nabla$ the quantum impulsion, which is an unbounded operator on $L^2$, and by $M_0$ the common total mass of the densities in both the quantum and the classical setting
	\begin{align*}
		M_0 := \Tr{\op} = \iintd f(t,x,v)\dd x\dd v.
	\end{align*}
	
	Our first result states that if the spatial density is concentrated enough, then the \textit{Eulerian moments} $\tr(\n{x-t\opp}^n\op)$ are bounded globally in time.

	\begin{thm}\label{th:propag_L}
		Let $d\geq 3$, $n\in 2\N\setminus\{0\}$, $r\in[1,\infty]$ and define $\fb_n := \frac{nr'+d}{n+1}$. Assume 
		\begin{equation}\label{eq:cond_K_1}
			\nabla K\in L^{\fb,\infty} \text{ with } \fb\in\(\max\!\(\tfrac{d}{3},\fb_4,\fb_n\)\!,\tfrac{d}{2}\),
		\end{equation}
		and let $\op$ be a solution of the \ref{eq:Hartree} equation with initial condition
		\begin{equation*}
			\op^\init\in\L^r\cap\L^1_+(1+\n{x}^n+\n{\opp}^n).
		\end{equation*}
		Then there exists an explicit constant $\C>0$ depending on $M_0$, $\tr(\n{\opp}^n\op^\init)$, $\Nrm{\nabla K}{L^{\fb,\infty}}$, $\Nrm{\op^\init}{\L^r}$ and not on $\hbar$, such that if
		\begin{equation}\label{eq:small_moments}
			\tr(\n{x}^n\op^\init) < \C,
		\end{equation}
		then
		\begin{equation*}
			\tr(\n{x-t\opp}^n\op) \in L^\infty(\R_+),
		\end{equation*}
		uniformly in $\hbar$.
	\end{thm}
	
	\begin{remark}\label{rem:ex}
		The theorem applies in particular in the case of interaction kernels $K$ with a singularity like the Coulomb interaction. For example for any $\eps>0$
		\begin{equation*}
			K(x) = \frac{\pm 1}{\n{x}} \mathds{1}_{\n{x}\leq 1} + \frac{\pm 1}{\n{x}^{1+\eps}} \mathds{1}_{\n{x}>1}.
		\end{equation*}
		An other interesting case of application is the case of the Yukawa potential that is commonly used as an approximation in the case when there are particles with positive and negative charge, and which is of the form
		\begin{equation*} 
				K(x) = \frac{e^{-\n{x}/\lambda_D}}{\n{x}},
		\end{equation*}
		where $\lambda_D>0$ is the Debye length, which represents the characteristic size of the interaction.
	\end{remark}
	
	\begin{remark}
		An other good example of potentials verifying the assumptions of the theorem are potentials of the form
		\begin{equation}\label{eq:K_homogene}
			K(x) = \frac{\pm 1}{\n{x}^a},
		\end{equation}
		with $a = \frac{d}{\fb} - 1 \in (1, \frac{8}{7})$ when $d=3$. In dimension $d=4$, $d=5$ and $d\geq 6$, one can even better take respectively $a\in (1, \frac{3}{2})$, $a\in (1, \frac{16}{9})$ and $a\in (1, 2)$. Of course, regular potentials also enter the scope of this first theorem as long as they decay sufficiently at infinity, so that we can take for example
		\begin{equation*}
			K(x) = \frac{\pm 1}{1+\n{x}^a},
		\end{equation*}
		for any $a > 1$.
	\end{remark}

	\begin{remark}
		Since $\op\in \L^1_+$, it is an Hilbert-Schmidt operator that can be written as a integral operator of kernel $\r(x,y)$ and it can also be diagonalized by the spectral theorem. Hence, we can write for any $\varphi\in L^2$
		\begin{equation}\label{eq:diagonalization}
			\op\,\varphi(x) = \intd \r(x,y)\,\varphi(y)\dd y = \sumj \lambda_j\ket{\psi_j}\bra{\psi_j}\varphi\rangle,
		\end{equation}
		where $(\psi_j)_{j\in J}\in (L^2)^J$ with $J\subset \N$ is an orthogonal basis. The space density can then be written
		\begin{equation*}
			\rho(x) := \r(x,x) = \sumj \lambda_j \n{\psi_j(x)}^2,
		\end{equation*}
		and the space moments 
		\begin{align*}
			\Tr{\n{x}^n\op} &= \intd \rho(x) \n{x}^n\d x.
		\end{align*}
		For even integers $n\in2\N$, the velocity moments can be written
		\begin{align}\label{def_moments}
			\Tr{\n{\opp}^{n}\op} &= \hbar^n \sumj \lambda_j \intd \n{\nabla^\frac{n}{2}\psi_j}^2.
		\end{align}
	\end{remark}
	
	\begin{remark}\label{rem:existence}
		Notice that the existence theory for both \ref{eq:Hartree} and \ref{eq:Vlasov}~equations is already quite well understood, see for example \cite{ginibre_class_1980, ginibre_global_1985, hayashi_smoothing_1989, brezzi_three-dimensional_1991, illner_global_1994, castella_l2_1997} for the \ref{eq:Hartree}~equation and \cite{lions_propagation_1991,schaeffer_global_1991,pfaffelmoser_global_1992} for the \ref{eq:Vlasov}~equation. For more singular potentials than the Coulomb potential, remark that by the real interpolation definition of Lorentz spaces, our hypothesis \eqref{eq:cond_K_1} on $\nabla K$ implies
		\begin{align*}
			\nabla K \in L^p + L^q \text{ for any } (p,q) \text{ such that } 0 < p < \fb < q < \infty,
		\end{align*}
		where one can take $p =\frac{2d+8}{d+8}$ since $\fb > \fb_4$. Moreover, by Sobolev embeddings, since $\fb > \frac{d}{3}$, 
		\begin{align*}
			K \in L^{r_0} + L^\infty \text{ with } r_0 > \tfrac{d}{2},
		\end{align*}
		and also with $r_0 > \tfrac{d+4}{4}$ when $d\geq 3$, since $\fb>\fb_4$. Therefore, our assumptions implies hypotheses (90) and (91) in \cite{lions_sur_1993} and so the existence of solutions for both equations. Remark that, as in our previous paper \cite{lafleche_propagation_2019}, we are not trying to prove here the propagation of regularity for $\psi$. In particular the global in time propagation for $\hbar=1$ of the multi-Sobolev norms defined by Formula~\eqref{def_moments} is proved in \cite[Appendix]{castella_l2_1997} in the case of the Coulomb potential, where they are denoted by $\Nrm{\Psi}{H^n(\lambda)}$ for $\Psi = \(\psi_j\)_{j\in J}$. The same analysis can be performed for our class of potentials. Thus, since we assume initially bounded velocity moments of order $n\geq 2$, this implies that our solutions will always satisfy $\Psi\in L^\infty_\loc((0,T),H^n(\lambda))$ but with a bound a priori not uniform in $\hbar$. Hence, the difficulty lies in the fact to obtain $\hbar$ independent bounds, which prevent for example to estimate separately each part of the commutator appearing in \ref{eq:Hartree} equation.
	\end{remark}
	
	We can state the analogue of this theorem for solutions of the \ref{eq:Vlasov}~equation
	\begin{prop}\label{prop:classic_propag_L}
		Let $d\geq 3$, $n\in 2\N$, $r\in[1,\infty]$ and assume $\nabla K$ verifies Condition~\eqref{eq:cond_K_1}. Let $f$ be a solution of the \ref{eq:Vlasov}~equation with nonnegative initial condition
		\begin{equation*}
			f^\init\in L^r_{x,\xi}\cap L^1(1+\n{x}^n+\n{\xi}^n).
		\end{equation*}
		Then there exists an explicit constant $\C > 0$ such that if
		\begin{equation*} 
			\iintd f^\init(x,\xi)\n{x}^n\d x\dd\xi < \C,
		\end{equation*}
		then
		\begin{equation*}
			\iintd f(\cdot,x,\xi)\n{x-t\xi}^n\d x\dd\xi \in L^\infty(\R_+).
		\end{equation*}
	\end{prop}
	
	\begin{remark}
		For the \ref{eq:Vlasov} equation, contrarily to the \ref{eq:Hartree} equation, we generally not have strong solutions in our setting (the velocity moments are not derivatives in the classical case). However, we still have global existence of renormalized solutions (see \cite{diperna_solutions_1988,diperna_global_1988}).
	\end{remark}

	We can use the first theorem to obtain good estimates on the space and velocity moments and on the spatial density that do not depend on $\hbar$.
	
	\begin{thm}\label{th:regu_VP}
		Let $d\geq 3$, $r\in[d',\infty]$, $n\in (2\N)\backslash\{0,2\}$ and assume 
		\begin{equation}\label{eq:cond_K_2}
			\nabla K\in L^{\fb,\infty} \text{ with } \fb\in\(\max\!\(\fb_4,\tfrac{d}{3}\)\!,\tfrac{d}{2}\),
		\end{equation}
		and let $\op$ be a solution of the \ref{eq:Hartree} equation with initial condition
		\begin{equation*}
			\op^\init\in \L^r\cap\L^1_+(1+\n{x}^4+\n{\opp}^n),
		\end{equation*}
		for a given even integer $n\geq 4$. Then there exists a constant $\C$ depending on $M_0$, $\Tr{\n{\opp}^4\op^\init}$, $\Nrm{\nabla K}{L^{\fb,\infty}}$, and $\Nrm{\op^\init}{\L^r}$ such that if
		\begin{equation*}
			\tr(\n{x}^4\op^\init) \leq \C,
		\end{equation*}
		then there exists $c_n = c_{d,n,r} \geq 0$ and $C>0$ depending on the initial conditions such that
		\begin{align}\label{eq:asympt_Mn}
			\tr(\n{\opp}^n\op) &\leq C \weight{t}^{c_n}
			\\\label{eq:asympt_Nn}
			\tr(\n{x}^n\op) &\leq C \weight{t}^{n\(c_n+1\)}
			\\\label{eq:asympt_rho}
			\Nrm{\rho}{L^p} &\leq C \weight{t}^{-d/p'},
		\end{align}
		where $p' = r' + \frac{d}{4}$ and $\weight{t} = \sqrt{1+t^2}$. Moreover, if $\Tr{\n{x}^n\op^\init}$ is sufficiently small, then we can get more precise estimates
		\begin{align*}
			\tr(\n{\opp}^n\op) &\leq C
			\\
			\tr(\n{x}^n\op) &\leq C \weight{t}^{n}
			\\
			\Nrm{\rho}{L^{p_n}} &\leq C \weight{t}^{-d/p_n'},
		\end{align*}
		where $p_n' = r' + \frac{d}{n}$.
	\end{thm}
	
	We can once more state the analogue result for the \ref{eq:Vlasov} equation.
	\begin{prop}\label{prop:classic_regu_VP}
		Let $d\geq 3$, $r\in[d',\infty]$, $n\in (2\N)\backslash\{0,2\}$ and assume $K$ verifies~\eqref{eq:cond_K_2}. Let $f$ be a solution of the \ref{eq:Vlasov}~equation with nonnegative initial condition
		\begin{equation*}
			f^\init\in L^r_{x,\xi}\cap L^1_{x,\xi}(1+\n{x}^4+\n{\xi}^n),
		\end{equation*}
		for a given even integer $n\geq 4$. Then there exists $\C > 0$ such that if
		\begin{equation*}
			\iintd f^\init(x,\xi)\n{x}^4\d x\dd\xi \leq \C,
		\end{equation*}
		then there exists $c_n = c_{d,n,r} > 0$ and $C>0$ depending on the initial conditions such that
		\begin{align*} 
			\iintd f(t,x,\xi)\n{\xi}^n\d x\dd\xi &\leq C \weight{t}^{c_n}
			\\ 
			\iintd f(t,x,\xi)\n{x}^n\d x\dd\xi &\leq C \weight{t}^{n(c_n+1)}
			\\ 
			\Nrm{\rho_f}{L^p} &\leq C \weight{t}^{-d/p'},
		\end{align*}
		where $p' = r' + \frac{d}{4}$. As in the previous theorem, one can take $c_n=0$ and $p'=r'+\frac{d}{n}$ if the initial space moments of order $n$ are small.
	\end{prop}
	
	Before stating the result about the semiclassical limit, we recall the definition of the semiclassical Wasserstein-(Monge-Kantorovitch) distance introduced by Golse and Paul in \cite{golse_schrodinger_2017}. We say that $\gam\in L^1(\R^{2d},\PP)$ is a semiclassical coupling between a classical kinetic density $f\in L^1\cap\P(\R^{2d})$ and a density operator $\op\in\PP$ and we write $\gam\in\C_\hbar(f,\op)$ when 
	\begin{align*}
		\Tr{\gam(z)} &= f(z)
		\\
		\intdd \gam(z)\dd z &= \op.
	\end{align*}
	Then we define the semiclassical Wasserstein-(Monge-Kantorovich) pseudo-distance in the following way
	\begin{equation}\label{def:Wh}
		\Wh(f,\op) := \(\inf_{\gam\in\C_\hbar(f,\op)}\intdd \Tr{\mathbf{c}_\hbar(z)\gam(z)}\d z\)^\frac{1}{2},
	\end{equation}
	where $\mathbf{c}_\hbar(z)\varphi(y) = \(|x-y|^2 + |\xi-\opp|^2\)\varphi(y)$, $z = (x,\mathbf{\xi})$ and $\opp = -i\hbar\nabla_y$. This is not a distance but it is comparable to the classical Wasserstein distance $W_2$ between the Wigner transform of the quantum density operator and the normal kinetic density, in the sense of the following theorem.
	\begin{theorem}[Golse \& Paul \cite{golse_schrodinger_2017}]\label{th:comparaison_W2}
		Let $\op\in\PP$ and $f\in\P(\R^{2d})$ be such that 
		\begin{equation*}
			\intdd f(x,\xi)\(\n{x}^2+\n{\xi}^2\)\d x\dd\xi < \infty.
		\end{equation*}
		Then one has $\Wh(f,\op)^2 \geq d\hbar$ and for the Husimi transform $\fht$ of $\op$, it holds 
		\begin{equation}\label{eq:comparaison_W2}
			W_2(f,\fht)^2 \leq \Wh(f,\op)^2 + d\hbar.
		\end{equation}
	\end{theorem}
	
	See \cite{golse_wave_2018} for more results about this pseudo-distance and the definition of the Husimi transform.
	
	Our last theorem uses these results to obtain the semiclassical limit. We also recall the following theorem which will gives us our assumptions on the classical solution of the \ref{eq:Vlasov}~equation.
	\begin{theorem}[Lions \& Perthame \cite{lions_propagation_1991}, Loeper \cite{loeper_uniqueness_2006}]\label{th:VP_classique}
		Assume $f^\init\in L^\infty_{x,\xi}(\R^6)$ verify
		\begin{equation}\label{hyp:VP_moment}
			\iint_{\R^6} f^\init\n{\xi}^{n_0}\d x\dd\xi < C \text{ for a given } n_0 > 6,
		\end{equation}
		and for all $R>0$,
		\begin{equation}\label{hyp:VP_intg}
			\supess_{(y,w)\in\R^6}\{f^\init(y+t\xi,w), |x-y|\leq Rt^2, |\xi-w|\leq Rt\} \in L^\infty_\loc(\R_+,L^\infty_xL^1_\xi).
		\end{equation}
		Then there exists a unique solution to the \ref{eq:Vlasov}~equation with initial condition $f_{t=0} = f^\init$. Moreover, in this case, the spatial density verifies $\rho_f\in L^\infty_\loc(\R_+,L^\infty).$
	\end{theorem}
	
	\begin{thm}\label{th:CV_VP}
		Let $d\geq 3$ and assume 
		\begin{align*}
			\nabla K\in B^1_{1,\infty} \cap L^\fb &\text{ with } \fb\in\(\max\!\(\fb_4,\tfrac{d}{3}\)\!,\tfrac{d}{2}\).
		\end{align*}
		Let $\op$ be a solution of the \ref{eq:Hartree} equation with initial condition $\op^\init\in\PP$ verifying 
		\begin{align*}
			\op^\init &\in \L^\infty \cap  \L^1_+(1+\n{\opp}^{n_1})
			\\
			\forall \ii\in\Int{1,d},\, \opp_\ii^{n_0}\op^\init &\in \L^\infty,
		\end{align*}
		where $\opp_\ii := -i\hbar\partial_\ii$, $\Int{1,d} = [1,d]\cap\N$, and $(n_0,n)\in(2\N)^2$ is such that
		\begin{align*}
			n_0 &> d
			\\
			n &\geq \frac{d+\fb\(n_0-1\)}{\fb-1}.
		\end{align*}
		Let $f$ is a solution of the \ref{eq:Vlasov}~equation with initial condition $f^\init$ verifying the hypotheses of Theorem~\ref{th:VP_classique} and of mass $M_0=1$. Then there exists a constant $\C>0$ depending on $\tr(\n{\opp}^4\op^\init)$, $\Nrm{\nabla K}{L^{\fb,\infty}}$ and $\Nrm{\op^\init}{\L^\infty}$ such that if
		\begin{equation*}
			\tr(\n{x}^4\op^\init) \leq \C,
		\end{equation*}
		then there exists a constant $C>0$ depending on the initial conditions such that
		\begin{align}\label{eq:rho_bound}
			\Nrm{\rho(t)}{L^\infty} \leq C\, \weight{t}^{n_0\(1+\frac{c}{\fb'}\)}
		\end{align}
		where $c=c_n$ is given by \eqref{eq:asympt_Mn}. Again, if additionally $\Tr{\n{x}^n\op^\init}$ is also sufficiently small, then one can take $c=0$. Moreover, the following semiclassical estimate holds
		\begin{equation*}
			\Wh(f(t),\op(t)) \leq \Wh(f^\init,\op^\init)^{e^{\theta\lambda(t)}}\, e^{e^{\lambda(t)}},
		\end{equation*}
		with
		\begin{align*}
			\theta &= \frac{\sign(\ln\Wh(f^\init,\op^\init))}{2}
			\\
			\lambda(t) &= C^\init \weight{t}^{1+n_0\(1+\frac{c}{\fb'}\)}
			\\
			C^\init &= 1 + C\Nrm{\nabla K}{B^1_{1,\infty}} \sup_{t} \frac{\Nrm{\rho_f(t)}{L^\infty} + \Nrm{\rho(t)}{L^\infty}}{\weight{t}^{n_0(1+c/\fb')}},
		\end{align*}
		for some constant $C>0$ independent from the initial conditions.
	\end{thm}
	
	\begin{remark}
		Again, the additional assumption $\nabla K\in B^1_{1,\infty}$ is compatible with a kernel with a Coulomb singularity in dimension $d=3$ such as the one given in Remark~\ref{rem:ex}. However, higher local singularities such that the one from Equation~\eqref{eq:K_homogene} are not admissible for this result. This seems natural since even the uniqueness of solutions for the \ref{eq:Vlasov} equation is not known for such singular potentials.
	\end{remark}
	
	\begin{remark}
		This theorem implies a result of convergence in the classical Wasserstein distance at a rate $C_t\, \hbar^{C_t}$ as soon as the quantity $W_{2,\hbar}(f^\init,\op^\init)$ is initially smaller than some power of $\hbar$. In particular, this implies weak convergence of the Wigner transform of the solution of the \ref{eq:Hartree} equation to the solution of the \ref{eq:Vlasov} equation. Remark that by \cite[Theorem~2.4]{golse_schrodinger_2017}, $W_{2,\hbar}(f^\init,\op^\init)$ is always larger than $\sqrt{d\hbar}$. The fact that $W_{2,\hbar}$ can be controlled by the classical Wasserstein distance up to an error term $\sqrt{d\hbar}$ can be proved for example when the initial states are superposition of coherent states and this leads to results that can be written uniquely in term of the classical Wasserstein distance as in \cite[Section~7]{lafleche_propagation_2019}.
	\end{remark}

\section{Free Transport}
	
	We want to use the time decay properties of the kinetic free transport equation which writes for $f = f(t,x,\xi)$
	\begin{equation*}
		\partial_t f + \xi\cdot\nabla_x f = 0.
	\end{equation*}
	In quantum mechanics, free transport is given by the free Schrödinger equation
	\begin{equation*}
		i\hbar\,\partial_t \psi = H_0 \psi,
	\end{equation*}
	where $\hbar = \frac{h}{2\pi}$ and $H_0 = -\frac{\hbar^2\Delta}{2}$ which can be written $H_0 = \frac{\n{\opp}^2}{2}$ with the notation $\opp = -i\hbar\nabla$. The solution corresponding to the initial condition $\psi^\init$ can be written $\T_t\psi^\init$ where the semigroup $\T_t$ is given by
	\begin{equation}\label{eq:semigroup_schrodinger}
		\T_t\psi = e^{-itH_0/\hbar}\psi = \frac{e^{-i\pi\n{x}^2/(h t)}}{(ih t)^{d/2}} * \psi.
	\end{equation}
	The corresponding equation for density operators $\op\in\PP$ is
	\begin{equation}\label{eq:liouville}
		i\hbar\,\partial_t \op = [H_0,\op],
	\end{equation}
	whose solution is $\sfS_t\op^\init$ where the semigroup $\sfS_t$ is defined by
	\begin{equation}\label{eq:semigroup_op}
		\sfS_t\op := \T_t\op \T_{-t}.
	\end{equation}
	As it can be easily noticed, it holds $\T_t^* = \T_t^{-1} = \T_{-t}$ and for any $(\op_1,\op_2)\in\PP^2$, $\sfS(\op\op_2) = \sfS(\op) \sfS(\op_2)$. Moreover, a straightforward computation shows that
	\begin{equation}\label{eq:schr_semigroup_prop}
		\begin{split}
			\sfS_t \opp &= \opp
			\\
			\sfS_t x &= x-t\opp.
		\end{split}
	\end{equation}
	By the spectral theory, it implies that $\sfS_t(f(x)) = f(x-t\opp)$ for any nice function~$f$. By analogy, we can define the operator of translation of the impulsion $\opp$ by
	\begin{align}\nonumber
		\Tx_t \psi(x) &:= e^\frac{-\pi\n{x}^2t}{ih} \psi(x) = G_t(x)\psi(x)
		\\\label{def:S_tilde}
		\Sx_t \op &:= \Tx_t\op \Tx_{-t},
	\end{align}
	which verifies the equation
	\begin{equation*}
		i\hbar\,\partial_t (\Sx_t\op) = \com{\frac{-\n{x}^2}{2}, \Sx_t\op},
	\end{equation*}
	and the two following relations
	\begin{equation}\label{eq:x_move}
		\begin{split}
			\Sx_t x &= x
			\\
			\Sx_t \opp &= \opp - t x.
		\end{split}
	\end{equation}

	We recall the quantum kinetic interpolation inequality that was already used in \cite[Theorem~6]{lafleche_propagation_2019}. For $k\in2\N$ we define
	\begin{align*}
		\rho_k &:= \sum_{j\in J}\lambda_j|\opp^\frac{k}{2}\psi_j|^2 = \Diag{\opp^\frac{k}{2}\op\cdot\opp^\frac{k}{2}},
	\end{align*}
	and for $r\geq 1$ and $0\leq k\leq n$, we define the exponent $p_{n,k}$ by its Hölder conjugate
	\begin{equation}\label{def:exposants}
		p_{n,k}' = \(\frac{n}{k}\)' p'_n \text{ with } p'_n = \(r' + \frac{d}{n}\).
	\end{equation}
	Then the following inequality holds
	\begin{prop}
		Let $(k,n)\in (2\N)^2$ be such that $k\leq n$ and $\op\in\L^1(|\op|^n)\cap\L^r_+$ for a given $r\in[1,\infty]$. Then there exists $C = C_{d,r,n,k} > 0$ such that
		\begin{align}\label{eq:Lieb_Thirring_k}
			\Nrm{\rho_k}{L^{p_{n,k}}} &\leq C \Tr{\n{\opp}^n\op}^{1-\theta} \Nrm{\op}{\L^r}^{\theta},
		\end{align}
		where $\theta = \frac{r'}{p_{n,k}'}$.
	\end{prop}
	
	From this result, we can get an inequality with an additional time decay if we replace the velocity moments by the Eulerian moments $\Tr{\n{x-t\opp}^n\op}$.
	
	\begin{cor}\label{cor:LT_transporte}
		Let $n\in 2\N$, $r\in[1,\infty]$, $p' := r' + \frac{d}{n}$ and $\theta = \frac{r'}{p'}$. Then
		\begin{equation*}
			\Nrm{\rho}{L^p} \leq \frac{1}{t^{d/p'}}\Tr{\n{x-t\opp}^n \op}^{1-\theta} \Nrm{\op}{\L^r}^\theta.
		\end{equation*}
	\end{cor}
	
	\begin{demo}[ of Corollary~\ref{cor:LT_transporte}]
		We just remark that by Formula~\eqref{eq:x_move}, we get
		\begin{align*}
			t^{-n}\Tr{\n{x-t\opp}^n \op} &= \Tr{|\opp - x/t|^n\op} 
			\\
			&= \Tr{\Sx_{1/t}(\n{\opp}^n)\op} = \Tr{\n{\opp}^n\Sx_{-1/t}(\op)}.
		\end{align*}
		Moreover, since $\Tx$ is a unitary transformation, the following identities hold
		\begin{align*}
			\Diag{\Sx_t\op} &= G_t(x)\,\r(x,x)\,G_{-t}(x) = \rho(x)
			\\
			\Nrm{\Sx_{-1/t}\op}{\L^r} &= \Nrm{\op}{\L^r}.
		\end{align*}
		Then by the interpolation Inequality~\eqref{eq:Lieb_Thirring_k} we get
		\begin{align*}
			\Nrm{\rho}{L^p} &= \Nrm{\Diag{\Sx_{-1/t}\op}}{L^p}
			\\
			&\leq \Tr{\n{\opp}^n\Sx_{-1/t}(\op)}^{1-\theta} \Nrm{\Sx_{-1/t}\op}{\L^r}^\theta
			\\
			&\leq t^{-n(1-\theta)} \Tr{\n{x-t\opp}^n \op}^{1-\theta} \Nrm{\op}{\L^r}^\theta.
		\end{align*}
		Finally, we remark that $n(1-\theta) = \frac{n(r'+d/n-r')}{r'+d/n} = \frac{d}{p'}$ to get the result.
	\end{demo}

\section{Propagation of moments}

\subsection{Classical case.}

	In this section, we define the classical Eulerian, velocity and space moments by
	\begin{align*}
		L_n &:= \iintd f(t,x,\xi)\n{x-t\xi}^n\d x\dd \xi
		\\
		M_n &:= \iintd f(t,x,\xi)\n{\xi}^n\d x\dd \xi
		\\
		N_n &:= \iintd f(t,x,\xi)\n{x}^n\d x\dd \xi.
	\end{align*}
	
	\begin{prop}[Classical large time estimate]\label{prop:classic_estim_t_long}
		Let $(r,\fb)\in [1,\infty]\times[\fb_n,\infty]$, $\nabla K\in L^{\fb,\infty}$ and $f\in L^\infty(\R_+,L^r_{x,\xi}\cap L^1_{x,\xi})$ be a nonnegative solution of \ref{eq:Vlasov}~equation. Then for any $n\in 2\N$, there exists a constant $C = C_{d,r,n} > 0$ such that
		\begin{equation*}
			\n{\ddt{L_n}} \leq C \Nrm{\nabla K}{L^{\fb,\infty}} M_0^{\Theta_0}\Nrm{f^\init}{L^r_{x,\xi}}^\frac{r'}{\fb}\ \frac{L_n^{1+\frac{a}{n}}(t)}{t^a},
		\end{equation*}
		where $a = \frac{d}{\fb} - 1$ and $\Theta_0 = 1 - \frac{a}{n} - \frac{r'}{\fb}$.
	\end{prop}

	\begin{demo}
		We write $f = f(t,x,\xi)$, $E = E(x)$ and we compute
		\begin{align*}
			\ddt{L_n} &= \iintd \n{x-t\xi}^n (-\xi\cdot\nabla_x f - E\cdot\nabla_\xi f) - \n{x-t\xi}^{n-2}(x-t\xi)\cdot\xi f\dd x\dd \xi
			\\
			&= -t\iintd f \n{x-t\xi}^{n-2} (x-t\xi)\cdot E\dd \xi\dd x.
		\end{align*}
		By Hölder's inequality, we deduce for $t\geq 0$ and any $p\in(1,\infty)$
		\begin{align*}
			\n{\ddt{L_n}} &\leq t \Nrm{\intd f\n{x-t\xi}^{n-1}\d \xi}{L^p_x} \Nrm{E}{L^{p'}}
			\\
			&\leq t^n \Nrm{\intd f\n{\tfrac{x}{t}-\xi}^{n-1}\d \xi}{L^p_x} \Nrm{\nabla K*\rho_f}{L^{p'}}
			\\
			&\leq C_K\, t^n \Nrm{\intd f(t,x,\xi+\tfrac{x}{t})\n{\xi}^{n-1}\d\xi}{L^p_x} \Nrm{\intd f\dd\xi}{L^q},
		\end{align*}
		where we used Hardy-Littlewood-Sobolev's inequality with $C_K = \Nrm{\nabla K}{L^{\fb,\infty}}$ and $p\in(1,\infty)$ such that
		\begin{equation}\label{eq:b2}
			\frac{1}{p'} + \frac{1}{q'} = \frac{1}{\fb}.
		\end{equation}
		Then we want to use the classical kinetic interpolation inequality which tells that for $p'_{n,k} = \(\frac{n}{k}\)'\(r' + \frac{d}{n}\)$ and $g=g(x,\xi)\geq 0$, it holds
		\begin{equation}\label{eq:interp}
			\Nrm{\intd g(x,\xi) \n{\xi}^k\d\xi}{L^{p_{n,k}}} \leq C_{d,r,n}\(\iintd g(x,\xi)\n{\xi}^n\d\xi\dd x\)^{1-r'/p'_{n,k}} \Nrm{g}{L^r_{x,\xi}}^{r'/p'_{n,k}}.
		\end{equation}
		Since
		\begin{equation*}
			\frac{1}{p'_{n,n-1}} + \frac{1}{p'_n} = \frac{1}{r'+d/n} \(1-\frac{n-1}{n} + 1\) = \frac{n+1}{nr'+d} = \frac{1}{\fb_n} \geq \frac{1}{\fb},
		\end{equation*}
		we can choose $p\leq p_{n,n-1}$ and $q\leq p_{n,0}$ verifying \eqref{eq:b2}. Take $p := p_{n,n-1}$. Then $1< q \leq p_n$ and by interpolation
		\begin{equation*}
			\Nrm{\intd f\dd\xi}{L^q} \leq M_0^{1-\frac{p'_n}{q'}} \Nrm{\intd f\dd\xi}{L^{p_n}}^{p'_n/q'}.
		\end{equation*}
		Using the above inequality and then the interpolation inequality~\eqref{eq:interp} for $k=0$ and $k=n-1$ yields
		\begin{align*}
			\n{\ddt{L_n}} &\leq C_{d,r,n}\,C_K\, t^n\, M_0^{1-\frac{p'_n}{q'}} \Nrm{\intd f(t,x,\xi+\tfrac{x}{t})\n{\xi}^{n-1}\d \xi}{L^p_x} \Nrm{\intd f\dd\xi}{L^{p_n}}^{p'_n/q'}
			\\
			&\leq C_{d,r,n}\,C_K\, t^n\, M_0^{1-\frac{r'}{\fb} + \frac{1}{n}\(1-\frac{d}{\fb}\)} \Nrm{f}{L^r_{x,\xi}}^\frac{r'}{\fb} \(\iintd f(t,x,\xi+\tfrac{x}{t}) \n{\xi}^n \d\xi\rt)^{1+\frac{1}{n} \(\frac{d}{\fb} - 1\)}.
		\end{align*}
		With $a = \frac{d}{\fb} - 1$ and by a change of variable, we get
		\begin{align*}
			\n{\ddt{L_n}} &\leq C_{d,r,n} C_K t^n M_0^{1-\frac{r'}{\fb} - \frac{a}{n}} \Nrm{f}{L^r_{x,\xi}}^\frac{r'}{\fb} \(\iintd f(t,x,\xi)\n{\xi-\tfrac{x}{t}}^n\d\xi\)^{1+\frac{a}{n}}
			\\
			&\leq C_{d,r,n} C_K t^{-a} M_0^{1-\frac{r'}{\fb} - \frac{a}{n}} \Nrm{f}{L^r_{x,\xi}}^\frac{r'}{\fb} \(\iintd f(t,x,\xi)\n{x-t\xi}^n\d \xi\)^{1+\frac{a}{n}},
		\end{align*}
		which is the expected inequality.
	\end{demo}
	
	\subsection{Boundedness of Eulerian moments}
	
	We define $\tilde{\op} := \Sx_{-1/t}(\op)$ and for $n\in 2\N$
	\begin{align*}
		\tilde{\rho}_n &:= \Diag{\opp^{n/2}\tilde{\op} \cdot\opp^{n/2}}
		\\
		l_n &:= t^n \tilde{\rho}_n.
	\end{align*}
	We also introduce the following notations for the Eulerian, velocity and space moments
	\begin{align*}
		L_n &:= \Tr{\n{x-t\opp}^n\op}
		\\	
		M_n &:= \Tr{\n{\opp}^n\op}
		\\
		N_n &:= \Tr{\n{x}^n\op},
	\end{align*}
	as well as the corresponding moments $\tilde{M}_n$ and $\tilde{N}_n$ for $\tilde{\rho}$. In particular, since we have
	\begin{align*}
		L_n &= \tr(\n{x-t\opp}^n\op) = t^n \tr(|\opp-x/t|^n\op)
		\\
		&= t^n \tr(\n{\opp}^n\tilde{\op}) = t^n \tr(\opp^{n/2}\tilde{\op}\cdot \opp^{n/2}),
	\end{align*}
	we obtain with these notations $L_n = \intd l_n$, $M_n = \intd \rho_n$ and $N_n = \intd \rho(x) \n{x}^n\d x$.
	
	\subsubsection{Long time estimate} To obtain a differential inequality which will give us the long time behavior of the solution, we first need the following time dependent interpolation inequalities.
	\begin{prop}
		Let $0\leq k\leq n$ and $p_{n,k}' := \(\frac{n}{k}\)'p_n'$ and $p\in[1,p_{n,k}]$. Then for any $\alpha \leq k$, there exists a constant $C = C_{d,r,n,k} > 0$ such that
		\begin{align}\label{eq:ineq_interp_double}
			\Nrm{\rho_k}{L^p} &\leq C\, M_\alpha^{1-\theta_{n,k,\alpha}} M_n^{\theta_{n,k,\alpha}-\frac{r'}{p'}} \Nrm{\op}{\L^r}^\frac{r'}{p'}
			\\\label{eq:ineq_interp_double_t}
			\Nrm{l_k}{L^p} &\leq C\, t^{-d/p'} L_\alpha^{1-\theta_{n,k,\alpha}} L_n^{\theta_{n,k,\alpha}-\frac{r'}{p'}} \Nrm{\op}{\L^r}^\frac{r'}{p'},
		\end{align}
		where
		\begin{align*}
			\theta_{n,k,\alpha} &= \frac{p_{n,\alpha}'}{p'} + \frac{k-\alpha}{n-\alpha}.
		\end{align*}
	\end{prop}
	
	\begin{demo}
		By the kinetic interpolation inequality~\eqref{eq:Lieb_Thirring_k},
		\begin{equation*}
			\Nrm{\rho_k}{L^{p_{n,k}}} \leq C_{d,r,n,k} M_n^{1-\frac{r'}{p_{n,k}'}}\Nrm{\op}{\L^r}^\frac{r'}{p_{n,k}'}.
		\end{equation*}
		Therefore, since $p\leq p_{n,k}$, by interpolation between $L^p$ spaces, we get
		\begin{align*}
			\Nrm{\rho_k}{L^p} &\leq \Nrm{\rho_k}{L^1}^{1-\theta} \Nrm{\rho_k}{L^{p_{n,k}}}^\theta
			\\
			&\leq C M_k^{1-\theta} M_n^{\theta-\frac{r'}{p'}} \Nrm{\op}{\L^r}^\frac{r'}{p'},
		\end{align*}
		where $\theta = \theta_{n,k,k} = \frac{p_{n,k}'}{p'}$ and we used the fact that $\Nrm{\rho_k}{L^1} = M_k$. It already proves Inequality~\eqref{eq:ineq_interp_double} for $k=\alpha$. Since $k\in[\alpha,n]$, we can also bound $M_k$ in the following way
		\begin{equation*}\label{eq:interp_Mk}
			M_k \leq M_\alpha^{1-\frac{k-\alpha}{n-\alpha}}M_n^\frac{k-\alpha}{n-\alpha},
		\end{equation*}
		which yields Inequality~\eqref{eq:ineq_interp_double}. To get \eqref{eq:ineq_interp_double_t}, we follow the proof of Corollary~\ref{cor:LT_transporte}. Since $\Sx$ preserves the Schatten norms, we can write
		\begin{align*}
			\Nrm{\tilde{\op}}{\L^r} = \Nrm{\op}{\L^r}.
		\end{align*}
		Hence, by replacing $\op$ by $\tilde{\op}$ in the kinetic interpolation inequality~\eqref{eq:Lieb_Thirring_k} and multiplying by $t^k$, we obtain
		\begin{align*}
			\Nrm{l_k}{L^{p_{n,k}}} = t^k \Nrm{\tilde{\rho}_k}{L^{p_{n,k}}} &\leq C t^k \(\Tr{\n{\opp}^n \Sx_{-1/t}\op}\)^{1-\frac{r'}{p_{n,k}'}} \Nrm{\Sx_{-1/t}\op}{\L^r}^\frac{r'}{p_{n,k}'}
			\\
			& \leq C t^k \(\tr(|\opp-x/t|^n\op)\)^{1-\frac{r'}{p_{n,k}'}}\Nrm{\op}{\L^r}^\frac{r'}{p_{n,k}'}
			\\
			& \leq C t^{k-n+\frac{nr'}{p_{n,k}'}} L_n^{1-\frac{r'}{p_{n,k}'}}\Nrm{\op}{\L^r}^\frac{r'}{p_{n,k}'}.
		\end{align*}
		Next we remark that
		\begin{align*}
			k-n+\frac{nr'}{p_{n,k}'} &= k-n+\(1-\frac{k}{n}\)\frac{nr'}{r'+d/n} &=  -(n-k)\(\frac{d/n}{r'+d/n}\) = -\frac{d}{p_{n,k}'},
		\end{align*}
		and we deduce Inequality~\eqref{eq:ineq_interp_double_t} again by interpolation of $L^p$ between $L^1$ and $L^{p_{n,k}}$ and by interpolation of $L_k$ between $L_\alpha$ and $L_n$.
	\end{demo}
	
	\begin{prop}[Large time estimate]\label{prop:estim_t_long}
		Let $(r,\fb)\in [1,\infty]\times[\fb_n,\infty]$, $\nabla K\in L^{\fb,\infty}$ and $\op\in L^\infty(\R_+,\L^r \cap \L^1_+)$ be a solution of the \ref{eq:Hartree} equation. Then for any $n\in 2\N$, there exists a constant $C = C_{d,r,n} > 0$ such that
		\begin{equation*}
			\n{\ddt{L_n}} \leq C \Nrm{\nabla K}{L^{\fb,\infty}} M_0^{\Theta_0}\Nrm{\op^\init}{\L^r}^\frac{r'}{\fb}\ \frac{L_n^{1+\frac{a}{n}}(t)}{t^a},
		\end{equation*}
		where $a = \frac{d}{\fb} - 1$ and $\Theta_0 = 1 - \frac{a}{n} - \frac{r'}{\fb}$.
	\end{prop}
	
	\begin{demo}
		We first remark that by Formula~\eqref{eq:schr_semigroup_prop} and spectral theory, we deduce $\n{x-t\opp}^n = \sfS_t(\n{x}^n)$. Therefore by defining $H_0 := \tfrac{\n{\opp}^2}{2}$, by definition of $\sfS_t$
		\begin{equation*}
			i\hbar\,\partial_t(\sfS_t(\n{x}^n)) = \com{H_0,\sfS_t(\n{x}^n)} = \com{H_0,\n{x-t\opp}^n}.
		\end{equation*}
		Hence, by differentiating $L_n$ with respect to time, we obtain
		\begin{align*}
			i\hbar\,\partial_t L_n &= \Tr{\com{H_0,\n{x-t\opp}^n}\op + \n{x-t\opp}^n\com{H_0+V,\op}}
			\\
			&= \Tr{\com{H_0,\n{x-t\opp}^n}\op + \com{\n{x-t\opp}^n,H_0+V}\op}
			\\
			&= \Tr{\com{\n{x-t\opp}^n,V}\op}.
		\end{align*}
		Then we use the operator $\Sx_t$ of translation in the $x$ direction defined in \eqref{def:S_tilde}. By formulas~\eqref{eq:x_move} and spectral theory, we deduce that for any $t\in\R$, $\Sx_{t} V = V$. Therefore, we deduce
		\begin{align*}
			i\hbar\,\partial_t L_n &= t^n \Tr{\com{(|\opp-x/t|^n),V}\op}
			\\
			&= t^n \Tr{\com{\Sx_{1/t}(\n{\opp}^n),V}\op}
			\\
			&= t^n \Tr{\com{\Sx_{1/t}(\n{\opp}^n),\Sx_{1/t}(V)}\op}
			\\
			&= t^n \Tr{\Sx_{1/t}\!\(\com{\n{\opp}^n,V}\)\op}
			\\
			&= t^n \Tr{\com{\n{\opp}^n,V}\tilde{\op}}.
		\end{align*}
		As it has been proved in \cite[Equation~(38)]{lafleche_propagation_2019}, this expression can be bounded in the following way
		\begin{align*}
			\n{\Tr{\com{\n{\opp}^n,V}\tilde{\op}}} \leq C_K\,\hbar \,\tilde{M}_{n}^\frac{1}{2} \sup_{\n{a+b+c}= n/2-1} \Nrm{\tilde{\rho}_{2\n{a}}}{L^\alpha}^\frac{1}{2} \Nrm{\tilde{\rho}_{2\n{b}}}{L^\beta}^\frac{1}{2} \Nrm{\tilde{\rho}_{2\n{c}}}{L^\gamma}^\frac{1}{2},
		\end{align*}
		where $(a,b,c) \in (\N^d)^3$ are multi-indices with $\n{a} = a_1 + ... + a_d$ and
		\begin{align}\label{eq:b}
			\frac{2}{\fb} &= \frac{1}{\alpha'}+\frac{1}{\beta'}+\frac{1}{\gamma'}
			\\\nonumber
			C_K &= C_{d,n}\Nrm{\nabla K}{L^{\fb,\infty}}.
		\end{align}
		As in \cite[Proof of Theorem~3, Step~2]{lafleche_propagation_2019}, we remark that for the exponents $p_{n,k}$ defined in \eqref{def:exposants} and multi-indices such that $\n{a+b+c} = n/2-1$, we have
		\begin{align*}
			\frac{1}{p_{n,2\n{a}}'} + \frac{1}{p_{n,2\n{b}}'} + \frac{1}{p_{n,2\n{b}}'} = \frac{1}{p'_n}\(3 - \frac{2\n{a} + 2 \n{b} + 2 \n{c}}{n}\) = \frac{2\(n+1\)}{nr'+d} = \frac{2}{\fb_n}.
		\end{align*}
		Therefore, since $\fb \geq \fb_n$, we can find $(\alpha,\beta,\gamma)\in [1,p_{n,2\n{a}}]\times[1,p_{n,2\n{b}}]\times[1,p_{n,2\n{b}}]$ verifying \eqref{eq:b} and use the interpolation inequality~\eqref{eq:ineq_interp_double_t} for $\alpha = 0$. By the definition of $l_k$ and the fact that $L_0 = M_0$, we deduce
		\begin{align*}
			\partial_t L_n &\leq C_K\, t^{n-n/2-(\n{a}+\n{b}+\n{c})}\, L_n^\frac{1}{2} \sup_{\n{a+b+c}= n/2-1} \Nrm{l_{2\n{a}}}{L^\alpha}^\frac{1}{2} \Nrm{l_{2\n{b}}}{L^\beta}^\frac{1}{2} \Nrm{l_{2\n{c}}}{L^\gamma}^\frac{1}{2}
			\\
			&\leq C_K\, t \,L_n^\frac{1}{2} \sup_{\n{a+b+c}= n/2-1} \Nrm{l_{2\n{a}}}{L^\alpha}^\frac{1}{2} \Nrm{l_{2\n{b}}}{L^\beta}^\frac{1}{2} \Nrm{l_{2\n{c}}}{L^\gamma}^\frac{1}{2}
			\\
			&\leq \(C_{d,r,n} C_K M_0^{\Theta_0} \Nrm{\op}{\L^r}^{\Theta_1}\) t^{-a}\, L_n^{\Theta},
		\end{align*}
		where
		\begin{align*}
			a &= \frac{d}{2}\(\frac{1}{\alpha'}+\frac{1}{\beta'}+\frac{1}{\gamma'}\) -1 = \frac{d}{\fb} - 1
			\\
			\Theta_0 &= \frac{1}{2} \(3 - p_n'\(\frac{1}{\alpha'}+\frac{1}{\beta'}+\frac{1}{\gamma'}\) - \frac{2\n{a}+2\n{b}+2\n{c}}{n}\)
			\\&= 1 + \frac{1}{n} - \frac{p_n'}{\fb} = 1 - \frac{a}{n} - \frac{r'}{\fb}
			\\
			\Theta_1 &= \frac{r'}{2}\(\frac{1}{\alpha'} + \frac{1}{\beta'} + \frac{1}{\gamma'}\) = \frac{r'}{\fb}
			\\
			\Theta &= \frac{1}{2} \(1 + p_n'\(\frac{1}{\alpha'}+\frac{1}{\beta'}+\frac{1}{\gamma'}\) + \frac{2\n{a}+2\n{b}+2\n{c}}{n} - r'\(\frac{1}{\alpha'} + \frac{1}{\beta'} + \frac{1}{\gamma'}\)\)
			\\
			&= 1 + \frac{1}{n} \(\frac{d}{\fb} - 1\) = 1 + \frac{a}{n}.
		\end{align*}
		We conclude by recalling that $\Nrm{\op}{\L^r}=\Nrm{\op^\init}{\L^r}$ since the \ref{eq:Hartree} equation preserves the Schatten norm.
	\end{demo}
	
	\subsubsection{Short time estimate.} To prove the short time estimate, we will use the boundedness of $M_n$ and $N_n$ for short times to get the boundedness of $L_n$. To achieve this, we first need some lemmas to bound traces expressions with products of $x$ and $\opp$ by $M_n$ and $N_n$.
	\begin{lem}[Interpolation for weighted traces]\label{lem:interpolation_w}
		Let $0\leq k\leq n$ and $A=\n{\opp}$ or $A=\n{x}$. Then for any operator $\op\geq 0$ we have the following inequalities
		\begin{align}\label{eq:interpolation_w}
			\Tr{A^k\op} &\leq \Tr{A^n\op}^\frac{k}{n} \Tr{\op}^{1-\frac{k}{n}}.
		\end{align}
	\end{lem}
	
	\begin{proof}
		By our definition of the absolute value, for two positive operators we have $\n{AB}^2 = BA^2B$. Therefore, the lemma follows from Hölder's inequality for the trace and Araki-Lieb-Thirring inequality \cite{araki_inequality_1990} since 
		\begin{align*}
			\Tr{A^k\op} = \Nrm{A^\frac{k}{2}\op^\frac{1}{2}}{2}^2 \leq \Nrm{A^\frac{k}{2}\op^\frac{k}{2 n}}{\frac{2 n}{k}}^2 \Nrm{\op^{\frac{1}{2}\(1-\frac{k}{n}\)}}{\frac{2 n}{2 n-k}}^2 \leq \Nrm{A^\frac{n}{2}\op^\frac{1}{2}}{2}^\frac{2k}{n} \Tr{\op}^{1-\frac{k}{n}},
		\end{align*}
		and the right-hand side here is exactly the right-hand side of Inequality~\eqref{eq:interpolation_w}.
	\end{proof}
	
	In the more general case of mixed product of $x$ and $\opp$, for any $\ii\in\Int{1,d}^n$ we can define the set 
	\begin{equation*}
		\cZ_k^n(\ii) = \{z_{\ii_1} \dots z_{\ii_n}, \forall j\in\Int{1,n}, z_{\ii_j} = x_{\ii_j} \text{ or } z_{\ii_j} = \opp_{\ii_j}, \n{\{j,z_j = \opp_j\}} = k\}.
	\end{equation*}
	of operators $\opz$ consisting of a product of $k$ partial derivatives and $n-k$ multiplications by a coordinate of $x$, and look at the following quantities
	\begin{equation*}
		\Tr{\opz\,\op} = \Tr{z_{\ii_1} \dots z_{\ii_n}\op}.
	\end{equation*}
	We denote by $\cZ_k^n := \bigcup_{\ii\in\Int{1,d}^n} \cZ_k^n(\ii)$. Then we have the following Lemma.
	\begin{lem}\label{lem:mixed_interpolation}
		Let $n\in 2\N$, $k\in\Int{1,n-1}$ and
		\begin{align*}
			Z_k^n = \sup_{\opz\in\cZ_k^n} \n{\Tr{\opz\,\op}}.
		\end{align*}
		Then there exist an integer $\mathcal{I}\in\Int{0,n}$ and a constant $C>0$ depending only on $d$, $n$ and $k$ such that for any operator $\op\geq 0$
		\begin{align}\label{eq:Z_k_0}
			Z_k^n \leq  M_n^{\frac{k}{n}} N_n^{1-\frac{k}{n}} + C\sum_{\ii=0}^{\mathcal{I}} \(\hbar\,Z_{k-1}^{n-2}\)^{2^{-\ii}} M_n^{\(1-2^{-\ii}\)\frac{k}{n}} N_n^{\(1-2^{-\ii}\)\(1-\frac{k}{n}\)}.
		\end{align}
		Moreover, $Z^n_0 = N_n$ and $Z^n_n = M_n$. This implies the fact that that for any $\eps>0$, there exists $C_\eps>0$ such that
		\begin{align}\label{eq:Z_k_1}
			Z_k^n \leq \eps\,N_n + C_\eps \(M_n + \hbar^\frac{n}{2} M_0\).
		\end{align}
	\end{lem}
	
	\begin{remark}
		In the case $k=1$, since $Z^{n-2}_0 = N_{n-2} \leq N_n^\frac{n-2}{n} M_0^\frac{2}{n}$, this leads to the following inequalities
		\begin{align}\nonumber
			Z^n_1 &\leq  M_n^{\frac{1}{n}} N_n^{1-\frac{1}{n}} + C\, \sum_{\ii=0}^{\mathcal{I}} \hbar^{2^{-\ii}} M_0^{\frac{2^{1-\ii}}{n}} M_n^{\frac{1-2^{-\ii}}{n}} N_n^{1-\frac{1+2^{-\ii}}{n}}
			\\\label{eq:Z_k_1_bis}
			&\leq C\(M_n^{\frac{1}{n}} N_n^{1-\frac{1}{n}} + \hbar\, M_0^{\frac{2}{n}} N_n^{1-\frac{2}{n}}\)
		\end{align}
	\end{remark}
	
	\begin{remark}
		For $n\in 2\N$, by expanding $\Tr{\n{x-\opp}^n\op}$ with the formula
		\begin{align}\label{eq:expand_Ln}
			\Tr{\n{x-\opp}^{n}\op} &= \Tr{\n{x}^{n}\op} + \Tr{\n{\opp}^{n}\op} + \sum_{k=1}^{n-1} \sum_{\opz\in \cZ_k^{n}} C_\opz \tr(\opz\,\op),
		\end{align}
		using the above lemma and Young's inequality, and replacing $\hbar$ by $t\hbar$, we obtain for any $\eps>0$
		\begin{equation}\label{eq:quasi_convexity}
			\Tr{\n{x-t\opp}^{n}\op} \leq \(1+\eps\)\Tr{\n{x}^{n}\op} + C_\eps \Tr{\(\n{t\opp}^{n} + \n{\hbar t}^\frac{n}{2}\)\op}.
		\end{equation}
	\end{remark}
	
	\begin{proof}
		The idea of the proof is to reiterate commutation of operators $z_i$ and Hölder's inequality for the trace. We start with the simplest case $n=2$.
		\step{1. Case $n=2$}
		
		In this case it is sufficient to look at $\n{\Tr{x_\ii \opp_j\op}}$ for any $(i,j)\in\Int{1,d}^2$. By Hölder's and Young's inequalities, we get
		\begin{align}\label{eq:Z_k_2}
			\n{\Tr{x_\ii \opp_j\op}} &= \n{\Tr{\op^\frac{1}{2}x_\ii \opp_j\op^\frac{1}{2}}} \leq \Nrm{\op^\frac{1}{2}x_\ii}{2} \Nrm{\opp_j\op^\frac{1}{2}}{2} \leq N_2^\frac{1}{2} M_2^\frac{1}{2}.
		\end{align}
		
		\step{2. Case $n>2$}
		
		Let $\opz\in \cZ^n_k$. Since we have the following commutation relations for $j\neq k\in\Int{1,d}^2$
		\begin{align}\label{eq:commut_0}
			\com{x_j,\opp_k} &= \com{\opp_j,\opp_k} = \com{x_j,x_k} = \com{x_j,x_j} = 0
			\\\label{eq:commut_1}
			\com{x_j,\opp_j} &= i\hbar,
		\end{align}
		any commutation operation of the form \eqref{eq:commut_1} of two adjacent $z_j$ in $\Tr{\opz\,\op} = \Tr{z_{\ii_1} \dots z_{\ii_{n}}\op}$ will add a term of the form $\Tr{\tilde{\opz}\,\op}$ with $\tilde{\opz}\in\cZ^{n-2}_{k-1}$. Therefore, we will use the notation $\C_{n-2}$ in the following steps to denote a term of the form
		\begin{align*}
			\C_{n-2} = \hbar \sum_{\tilde{\opz}\in \cZ_{k-1}^{n-2}} C_{\tilde{\opz}} \Tr{\tilde{\opz}\,\op},
		\end{align*}
		for some constants $C_{\tilde{\opz}}>0$. Hence, we can write
		\begin{align*}
			\Tr{\opz\,\op} \leq \Tr{\opp^\alpha x^\beta\op} + \C_{n-2} \leq \Tr{\op^\frac{1}{2}\opp^\alpha x^\beta\op^\frac{1}{2}} + \C_{n-2}
		\end{align*}
		with $(\alpha,\beta)\in (\N^d)^2$ such that $\n{\alpha}=k$ and $\n{\beta}=n-k$, where we use the multi-index notations $\n{\alpha} = \alpha_1+\dots+\alpha_d$, $x^\beta = x_1^{\beta_1}\dots x_d^{\beta_d}$, and $\opp^\alpha = (-i\hbar)^{\n{\alpha}}\partial_{x_1}^{\alpha_1}\dots\partial_{x_1}^{\alpha_1}$.
		
		$\bullet$ If $\n{\alpha}\leq\frac{n}{2}$, then we can write $\beta = (\gamma,\delta)$ with $\n{\gamma} = \frac{n}{2}-\n{\alpha}\geq 0$ and $\n{\delta} = \frac{n}{2}$, so that by Hölder's inequality
		\begin{align*}
			\Tr{\opp^{\alpha} x^\beta\op} = \Tr{\op^\frac{1}{2}\opp^\alpha x^\gamma x^\delta\op^\frac{1}{2}} &\leq \Tr{\n{\op^\frac{1}{2}\opp^\alpha x^\gamma}^2}^\frac{1}{2} \Tr{\n{x^\delta\op^\frac{1}{2}}^2}^\frac{1}{2}
			\\
			&\leq \Tr{\opp^\alpha x^{2\gamma}\opp^\alpha\op}^\frac{1}{2} N_n^\frac{1}{2}
			\\
			&\leq \Tr{\opp^{2\alpha}x^{2\gamma}\op}^\frac{1}{2} N_n^\frac{1}{2} + \C_{n-2}^\frac{1}{2}\, N_n^\frac{1}{2}
		\end{align*}
		where we used the inequality $\sqrt{a+b}\leq \sqrt{a}+\sqrt{b}$. If $\n{\alpha}=\frac{n}{2}$, then $\n{\gamma}=0$ and $\Tr{\opp^{2\alpha}x^{2\gamma}\op} \leq M_n$. If $\n{\alpha}<\frac{n}{2}$, remark that $\Tr{\opp^{2\alpha}x^{2\gamma}\op} = \Tr{\opp^{\alpha_{(1)}}x^{\beta_{(1)}}\op}$ with $\alpha_{(1)}=2\,\alpha$ and $\beta_{(1)} = 2\gamma$, which is of the same form as the term $\Tr{x^\alpha\opp^\beta\op}$ we started from, but with $\n{\alpha_{(1)}} = 2 \n{\alpha} > \n{\alpha}$. We can continue repeating this process creating in this way a sequence $(\alpha_{(\ii)})_{\ii\in\{1,\dots,j\}}$ as long as $\n{\alpha_{(\ii-1)}}< \frac{n}{2}$, and with the property that $\n{\alpha_{(\ii)}} = 2^{\ii}\n{\alpha}$ is increasing and $\n{\beta_{(\ii)}} = n - \n{\alpha_{(\ii)}}$ is decreasing. If the last term of this sequence $\alpha_{(j)}$ verifies $\n{\alpha_{(j)}}=\frac{n}{2}$, then $\n{\beta_{(j)}}=\frac{n}{2}$ and we get
		\begin{align*}
			\Tr{\opz\,\op} &\leq \Tr{\opp^{\alpha_{(j)}}x^{\beta_{(j)}}\op}^\frac{1}{2^{j}} N_n^{1-\frac{1}{2^{j}}} + \sum_{\ii=0}^j \C_{n-2}^\frac{1}{2^\ii} N_n^{1-\frac{1}{2^\ii}}
			\\
			&\leq M_n^\frac{1}{2^{j+1}} N_n^{1-\frac{1}{2^{j+1}}} + \sum_{\ii=0}^{j+1} \C_{n-2}^\frac{1}{2^\ii} N_n^{1-\frac{1}{2^\ii}},
		\end{align*}
		where $\C_{n-2}$ depends on $\ii$.
		
		$\bullet$ If this is not the case, then we get $\n{\alpha_{(j)}}\in\(\frac{n}{2},n\)$, so we are in the case of terms of the form $\Tr{\opp^{\alpha} x^\beta\op}$ with $\n{\alpha}>\frac{n}{2}$. For such a term, we do the same reasoning as in the case $\n{\alpha}<\frac{n}{2}$ but replacing the role of $x$ and $\opp$. This gives us a sequence $(\alpha_{(\ii)},\beta_{(\ii)})_{\ii\in\Int{j,j_2}}$ with $\n{\beta_{(\ii)}}< \frac{n}{2}$ strictly increasing, $\n{\alpha_{(\ii)}}> \frac{n}{2}$ strictly decreasing and such that
		\begin{align*}
			\Tr{\opp^{\alpha_{(j)}}x^{\beta_{(j)}}\op} \leq \Tr{\opp^{\alpha_{(j_2)}}x^{\beta_{(j_2)}}\op}^\frac{1}{2^{j_2-j}} M_n^{1-\frac{1}{2^{j_2-j}}} + \sum_{\ii=1}^{j_2-j} \C_{n-2}^\frac{1}{2^\ii} M_n^{1-\frac{1}{2^\ii}}.
		\end{align*}
		One more time, the last term verifies either $\n{\alpha_{(j_2)}} = \frac{n}{2}$, either $\n{\alpha_{(j_2)}} < \frac{n}{2}$, in which case we continue to create the sequence as we did for $\ii\leq j$. 
		
		$\bullet$ If the sequence $\n{\alpha_{(\ii)}}_{\ii\in\N}$ never converges to $\frac{n}{2}$, there is a periodic orbit. More precisely, since the sequence $\(\n{\alpha_{(\ii)}}\)_{\ii\in\N}$ takes values in the finite set $\Int{1,n-1}\setminus\{\frac{n}{2}\}$, we deduce that it will come back twice at the same point, i.e. there exists $(\ii_0,\ii_1)\in\N^2$ such that $0 <\ii_0< \ii_1<n$, $\alpha_{(\ii_0)} = \alpha_{(\ii_1)}$ and 
		\begin{align}\label{eq:looping}
			\Tr{\opp^{\alpha_{(\ii_0)}}x^{\beta_{(\ii_0)}}\op} \leq \Tr{\opp^{\alpha_{(\ii_1)}}x^{\beta_{(\ii_1)}}\op}^{2^{-\ii_2}} M_n^{\theta_{\ii_2}} N_n^{\eps_{\ii_2}} + \sum_{\ii=1}^{\ii_2} \C_{n-2}^{2^{-\ii}} M_n^{\theta_{\ii}} N_n^{\eps_{\ii}},
		\end{align}
		where $\ii_2 = \ii_1-\ii_0$ and $(\theta_\ii,\eps_\ii)\in(0,1)^2$ are such that 
		\begin{equation*}
			\theta_\ii+\eps_\ii = 1 - 2^{-\ii} =: \tau_\ii.
		\end{equation*}
		From Inequality~\eqref{eq:looping}, using Young's inequality on the first term of the right-hand side and removing $2^{-\ii_2}\Tr{\opp^{\alpha_{(\ii_0)}}x^{\beta_{(\ii_0)}}\op}$ on both side yields
		\begin{align*}
			\tau_{\ii_2}\Tr{\opp^{\alpha_{(\ii_0)}}x^{\beta_{(\ii_0)}}\op} \leq \tau_{\ii_2}\, M_n^\frac{\theta_{\ii_2}}{\tau_{\ii_2}} N_n^\frac{\eps_{\ii_2}}{\tau_{\ii_2}} + \sum_{\ii=1}^{\ii_2} \C_{n-2}^{2^{-\ii}} M_n^{\theta_{\ii}} N_n^{\eps_{\ii}},
		\end{align*}
		and this leads to
		\begin{align*}
			\Tr{\opz\,\op} &\leq \Tr{\opp^{\alpha_{(\ii_0)}}x^{\beta_{(\ii_0)}}\op}^{2^{-\ii_0}} M_n^{\Theta_{\ii_0}} N_n^{\Eps_{\ii_0}} + \sum_{\ii=0}^{\ii_0} \C_{n-2}^{2^{-\ii}} M_n^{\Theta_{\ii}} N_n^{\Eps_{\ii}}
			\\
			&\leq M_n^{\Theta} N_n^{\Eps} + \sum_{\ii=0}^{\ii_1} \C_{n-2}^{2^{-\ii}} M_n^{\Theta_{\ii}} N_n^{\Eps_{\ii}},
		\end{align*}
		where $(\Theta_{\ii},\Eps_{\ii})= (\Theta_{\ii_0},\Eps_{\ii_0}) + 2^{-\ii_0}(\theta_{\ii-\ii_0},\eps_{\ii-\ii_0})$ if $\ii>\ii_0$ and $\Theta = \frac{2^{-\ii_0}\theta_{\ii_2}}{\tau_{\ii_2}}+\Theta_{\ii_0}$. In particular, since $\Theta_{\ii_1}=\Theta_{\ii_0}+2^{-\ii_0}\theta_{\ii_2}$, we deduce that
		\begin{equation*}
			\Theta = \Theta_{\ii_0} + \frac{1}{\tau_{\ii_2}} \(\Theta_{\ii_1}-\Theta_{\ii_0}\) = \Theta_{\ii_1} + \frac{2^{\ii_0}}{2^{\ii_1} - 2^{\ii_0}}\(\Theta_{\ii_1}-\Theta_{\ii_0}\) > \Theta_{\ii_1},
		\end{equation*}
		and similarly $\Eps = 1-\Theta > \Eps_{\ii_1}$.
		
		$\bullet$ In all the cases, we end up with an inequality of the form
		\begin{align}\label{eq:result_temp}
			\Tr{\opz\,\op} &\leq M_n^{\Theta} N_n^{\Eps} + \sum_{\ii=0}^{\mathcal{I}} \C_{n-2}^{2^{-\ii}} M_n^{\Theta_{\ii}} N_n^{\Eps_{\ii}},
		\end{align}
		where $(\Theta_{\ii})_{\ii\in\Int{0,\mathcal{I}}}$ and $(\Eps_{\ii})_{\ii\in\Int{0,\mathcal{I}}}$ are two increasing sequences with values in $[0,1)$ such that $\Eps_{\ii}+\Theta_{\ii} = \tau_{\ii}$, $\Theta\geq \Theta_{\mathcal{I}}$ and $\Eps \geq \Eps_{\mathcal{I}}$.
		
		\step{3. Rescaling} Replacing $\hbar$ by $\lambda\,\hbar$ for some $\lambda>0$ and dividing Formula~\eqref{eq:result_temp} by $\lambda^k$ yields
		\begin{align*}
			\Tr{\opz\,\op} \leq \lambda^{\Theta n-k} M_n^{\Theta} N_n^{1-\Theta} + \sum_{\ii=0}^{\mathcal{I}} \lambda^{\Theta_{\ii}n-\tau_{\ii}k}\,\C_{n-2}^{2^{-\ii}} M_n^{\Theta_{\ii}} N_n^{\tau_{\ii}-\Theta_{\ii}} .
		\end{align*}
		Taking $\lambda = N_n^\frac{1}{n} M_n^{-\frac{1}{n}}$, we arrive at the following formula
		\begin{align*}
			\Tr{\opz\,\op} \leq  M_n^{\frac{k}{n}} N_n^{1-\frac{k}{n}} + \sum_{\ii=0}^{\mathcal{I}} \C_{n-2}^{2^{-\ii}} M_n^{\tau_{\ii}\frac{k}{n}} N_n^{\tau_{\ii}\(1-\frac{k}{n}\)}.
		\end{align*}
		from which we deduce~\eqref{eq:Z_k_0}. Formula~\eqref{eq:Z_k_1} then follows by an induction on $n$, Lemma~\ref{lem:interpolation_w} and Young's inequality.
	\end{proof}
	
	To get a short time Eulerian moment estimate, we use \cite[Theorem~3]{lafleche_propagation_2019} which tells us that for any $n\in 2\N$ and $\fb>\max(\fb_4,\fb_n)$, there exists a time
	\begin{equation}\label{def:Tmax}
		T = T_{\Nrm{\nabla K}{L^{\fb,\infty}},\Nrm{\op^\init}{\L^r},M_0,M_n^\init,d,r,n},
	\end{equation}
	and a positive constant $m$ depending on $\nabla K$, $\Nrm{\op^\init}{\L^r}$, $M_0$, $M_n^\init$, $d$, $r$ and $n$ such that
	\begin{equation}\label{eq:borne_loc_M}
		\forall (k,t)\in [0,n]\times[0,T],\ M_k(t) \leq m.
	\end{equation}
	
	\begin{prop}[Short time estimate]\label{prop:estim_t_court}
		Assume $\hbar\in(0,1)$ and let $n\in2\N\backslash\{0\}$, $r\in[1,\infty]$,
		\begin{equation*}
			\nabla K\in L^{\fb,\infty} \text{ for } \fb\in(\max(\fb_4,\fb_n),\infty],
		\end{equation*}
		and $\op\in L^\infty([0,T],\L^r\cap \L^1_+(1+\n{x}^n+\n{\opp}^n))$ be a solution of the \ref{eq:Hartree} equation. Then for any $t\in[0,T]$ it holds
		\begin{equation*}
			L_n \leq 2^n \tr(\n{x}^n\op^\init) + C_{d,n,T}\,m\(1+\hbar^\frac{n}{2}\)t^\frac{n}{2},
		\end{equation*}
		where $T$ is given by \eqref{def:Tmax}.
	\end{prop}
	
	\begin{demo}
		We first remark that
		\begin{align*}
			[\n{x}^n,\n{\opp}^2] &= 2i\hbar\,\nabla(\n{x}^n)\cdot\opp - (-i\hbar)^2\Delta(\n{x}^n)
			\\
			&= ni\hbar\n{x}^{n-2}\(2x\cdot\opp - i\hbar\(d+n-2\)\).
		\end{align*}
		Therefore, recalling the notation $N_n := \tr(\n{x}^n\op)$ and using the fact that $\n{x}^n$ commutes with $V(x)$, we can compute
		\begin{align*}
			\ddt{N_n} &= \frac{1}{i\hbar} \Tr{\com{\n{x}^n,\frac{\n{\opp}^2}{2}+V}\op}
			\\
			&= \frac{n}{2} \Tr{\n{x}^{n-2}\(2 x\cdot\opp-i\hbar\(d+n-2\)\)\op}
			\\
			&= n \Tr{\n{x}^{n-2} x\cdot\opp\, \op} - \frac{n\,i\hbar\(d+n-2\)}{2}\Tr{\n{x}^{n-2}\op}.
		\end{align*}
		But since $N_n$ remains a real number, its derivative is the same as the derivative of its real part and we finally obtain
		\begin{align}\label{eq:dt_N_0}
			\ddt{N_n} &= n \Tr{\n{x}^{n-2} x\cdot\opp\,\op}.
		\end{align}
		This quantity can be bounded using Inequality~\eqref{eq:Z_k_1_bis}, leading to
		\begin{align}\label{eq:dt_N}
			\ddt{N_n} &\leq C\, \hbar\,M_0^\frac{2}{n}N_n^{1-\frac{2}{n}}+ C\,M_n^\frac{1}{n}N_n^{1-\frac{1}{n}}.
		\end{align}
		Since $M_n$ is uniformly bounded on $[0,T]$ by the bound \eqref{eq:borne_loc_M}, we deduce by Gronwall's Lemma that $N_n$ is also bounded uniformly on $[0,T]$. Therefore, since $\frac{n}{2}\leq n$, we can write $N_n^\frac{n-1}{n} \leq \C_T N_n^\frac{n-2}{n}$ on $[0,T]$ with $\C_T^n = \Nrm{N_n}{L^\infty(0,T)}$ depending only on the initial conditions. Thus, Inequality~\eqref{eq:dt_N} implies $\partial_t N_n \leq C_T\,N_n^\frac{n-2}{n}$ with $C_T = \hbar\, M_0^\frac{2}{n} + C\,\C_T\, m^\frac{1}{n}$. Therefore, using once more Gronwall's Lemma
		\begin{align}\label{eq:dt_N_2}
			N_n(t) &\leq \((N_n^\init)^\frac{2}{n} + \tfrac{2}{n}\,C_T\,t\)^\frac{n}{2} \leq 2^{\frac{n}{2}-1} \(N_n^\init + (C_T\,t)^\frac{n}{2}\).
		\end{align}
		Finally, by Inequality~\eqref{eq:quasi_convexity} with $1+\eps = 2^{1/2}$, and the fact that $t\leq T$, we obtain
		\begin{align*}
			L_n &\leq 2^\frac{n}{2}\,N_n + C_n \( t^nM_n + \n{\hbar t}^\frac{n}{2}M_0 \)
			\\
			&\leq 2^n \(N_n^\init + t^\frac{n}{2}\(C_T^\frac{n}{2}+C_n\,T^\frac{n}{2}\,m  + C_n\,\hbar^\frac{n}{2}\, M_0\)\),
		\end{align*}
		which yields the result.
	\end{demo}
	
	\subsubsection{Global estimate.} To prove the first theorem, we will now combine the short time estimate, which tells that $L_n$ is not growing fast initially, with the long time estimate which works only when $L_n$ is not to large after a short time. Since by assumption $L_n$ is small initially, the combination of these estimates will give us a global bound on $L_n$.
	
	\begin{demo}[ of Theorem~\ref{th:propag_L}]
		Since $\fb<\frac{d}{2}$, we have $a := \frac{d}{\fb} - 1 >1$. Thus, by Gronwall's Lemma and Proposition~\ref{prop:estim_t_long}, for any $t>\tau>0$ we obtain
		\begin{align*}
			L_n(t)^{-a/n} &\geq L_n(\tau)^{-a/n} + \frac{1}{A} \(\frac{1}{t^{a-1}} - \frac{1}{\tau^{a-1}}\)
			\\
			&\geq L_n(\tau)^{-a/n} -\frac{1}{A\tau^{a-1}},
		\end{align*}
		where
		\begin{equation*}
			A = \(1-\frac{1}{a}\)\frac{n}{C \Nrm{\nabla K}{L^{\fb,\infty}} M_0^{\Theta_0}\Nrm{\op^\init}{\L^r}^\frac{r'}{\fb}}.
		\end{equation*}
		Combining the above inequality with Proposition~\ref{prop:estim_t_court}, we know that there exists $T$ such that for any $\tau\in (0,T]$ and $t>0$, it holds
		\begin{equation*}
			L_n(t) \leq \(\(2^n N_n^\init + C_T \tau^\frac{n}{2}\)^{-\frac{a}{n}} -\frac{1}{A\tau^{a-1}}\)^{-n/a},
		\end{equation*}
		as soon as the right-hand side is positive and with $C_T = 	C_{T,N_n^\init,M_4^\init, M_0}$. Now we choose $\tau$ and $N^\init$ so that this positivity property holds. Since $\fb > \frac{d}{3}$, we have $2 < a' = \frac{a}{a-1}$. Thus we can define $\tau_0 := \min\!\(T,(AC_T^\frac{-a}{n})^\frac{2}{2-a}\)$ and $N := 2^{-n}\(A^\frac{n}{a}\tau_0^{\frac{n}{2a'}} - C_T \tau_0^\frac{n}{4}\)$. We remark that $N \geq 0$ since
		\begin{equation*}
			\tau_0 \leq (AC_T^\frac{-a}{n})^\frac{2}{2-a} \implies C_T \tau_0^{\frac{n}{4}-\frac{n}{2a'}} \leq A^\frac{n}{a} \implies N \geq 0. 
		\end{equation*}
		Taking $\tau = \tau_0$ and $N_n^\init < N$, we obtain that
		\begin{align*}
			C_{T,N_n^\init,M_4^\init, M_0} &:= (2^n N_n^\init + C_T \tau_0^\frac{n}{4})^{-\frac{a}{n}} -\frac{1}{A\tau_0^{a-1}} 
			\\
			&> (N + C_T \tau_0)^{-\frac{a}{n}} -\frac{1}{A\tau_0^{a-1}} = 0.
		\end{align*}
		We deduce that for any $t>0$
		\begin{equation*}
			L_n(t) < C_{T,N_n^\init,M_4^\init, M_0}^{-\frac{n}{a}},
		\end{equation*}
		which proves the result.
	\end{demo}
	
	\subsection{Application to the semiclassical limit}
	
	We fix now $\op\geq 0$ a solution of the \ref{eq:Hartree} equation and we assume $\hbar\leq 1$ (or equivalently, we do not write the $\hbar$ dependence of constants that are bounded when $\hbar\to 0$). We will now use the uniform in time estimate on $L_n$ and again the differential inequality for $N_n$ to obtain bounds on $N_n$ and $M_n$.
	
	To implement this strategy, first remark that by defining $\tilde{\op} := \Sx_{-1/t}\op$ and $\tilde{M}_n:=\Tr{\n{\opp}^n\tilde{\op}}$, then as in the proof of Corollary~\ref{cor:LT_transporte}, it holds
	\begin{equation}\label{eq:tilde_Mn}
		t^n\tilde{M}_n = t^n \Tr{\n{\opp}^n\Sx_{-1/t}(\op)} = \Tr{\n{x-t\opp}^n\op} = L_n,
	\end{equation}
	while $\Tr{\n{x}^n\tilde{\op}} = N_n$ and $\Tr{\tilde{\op}} = M_0$ are unchanged by this transformation. Similarly, we also have
	\begin{equation*}
		t^nM_n = \Tr{\n{x+t\opp}^n\tilde{\op}},
	\end{equation*}
	Expanding the right-hand side of this inequality as in Equation~\eqref{eq:expand_Ln} and then using Formula~\eqref{eq:tilde_Mn}, we obtain
	\begin{equation}\label{eq:bound_Mn_Ln}
		t^n M_n \leq N_n + L_n + C\,\sum_{k=1}^{n-1} t^k \,\tilde{Z}_k^n,
	\end{equation}
	where $\tilde{Z}_k^n = \sup_{\opz\in\cZ_k^n} \n{\Tr{\opz\,\tilde\op}}$. From this inequality and the result of Theorem~\ref{th:propag_L}, we obtain the following bounds.
	\begin{prop}\label{prop:propag_MN}
		Under the hypotheses of Theorem~\ref{th:propag_L}, it holds
		\begin{align*}
			N_n &\leq N_n^\init + C \(t^{n} + t^\frac{n}{2}\)
			\\
			M_n &\leq C,
		\end{align*}
		and for any $k\in\Int{0,n}$, $\tilde{Z}_k^n \leq C\,t^{-k}\weight{t}^{n-k}$, where the constants $C>0$ involved depends on $\Nrm{\nabla K}{L^{\fb,\infty}}$, $M_0$, $M_n^\init$, $\Nrm{\op^\init}{\L^r}$, $d$, $n$ and $r$.
	\end{prop}
	
	\begin{demo}
		Once again, we will proceed by induction on $n$.

		\step{1. Case $n=2$} In this case, conservation of energy together with the boundedness of the potential energy by $M_2$ as soon as $K\in L^{\frac{d+2\,r'}{4},\infty}+L^\infty$ implies that $M_2$ is uniformly bounded independently from $\hbar$ and $t$ (see e.g. \cite[Remark~3.1]{lafleche_propagation_2019} or \cite{lions_sur_1993}). By Sobolev's embedding, this is always true if $\nabla K\in L^{\fb_4}$.
		Then by Formula~\eqref{eq:dt_N_0} and Inequality~\eqref{eq:Z_k_2}, we deduce
		\begin{align*}
		 \ddt{N_2} = 2 \Tr{x\cdot\opp\,\op} \leq 2\,M_2^\frac{1}{2} N_2^\frac{1}{2},
		\end{align*}
		which by Gronwall's inequality yields
		\begin{align*}
			N_2(t) \leq \(N_2(0)^\frac{1}{2} + m_2 \,t\)^2,
		\end{align*}
		where $m_2 = \Nrm{M_2}{L^\infty(\R_+)}$. Then, by Formula~\eqref{eq:tilde_Mn}, we obtain
		\begin{equation*}
			\Tr{x\cdot\opp \,\tilde{\op}} \leq 2\, \tilde{M}_2^\frac{1}{2} \tilde{N}_2^\frac{1}{2} = 2\, t^{-1}\, L_2^\frac{1}{2} N_2^\frac{1}{2} \leq C\,t^{-1}\weight{t},
		\end{equation*}
		from which we easily deduce that $\tilde{Z}^2_1 \leq C\,t^{-1}\weight{t}$. Finally, $\tilde{Z}^2_0 = N_2 \leq C\weight{t}^2$ and $\tilde{Z}^2_2 = \tilde{M}_2 = L_2 \,t^{-2}$.
		
		\step{2. Case $n>2$}
		We go back to Equation~\eqref{eq:dt_N_0} which together with Inequality~\eqref{eq:Z_k_0} with $k=1$ and the fact that by our induction hypothesis $Z^{n-2}_0 = N_{n-2} \leq C \weight{t}^{n-2}$ implies that
		\begin{align}\nonumber
			\ddt{N_n} &\leq n\(M_n^{\frac{1}{n}} N_n^{1-\frac{1}{n}} + C\sum_{\ii=0}^{n} \(\hbar\weight{t}^{n-2}\)^{2^{-\ii}} M_n^{\frac{1-2^{-\ii}}{n}} N_n^{\(1-2^{-\ii}\)\(1-\frac{1}{n}\)}\)
			\\\label{eq:dt_N_3}
			&\leq n\,M_n^{\frac{1}{n}} N_n^{1-\frac{1}{n}} + C\(\weight{t}^{n-2} + \weight{t}^{\eps\(n-2\)} M_n^{\frac{1-\eps}{n}} N_n^{\(1-\eps\)\(1-\frac{1}{n}\)}\),
		\end{align}
		with $\eps = 2^{-n}$. Remark that we can replace $N_n$ by $N:=1+N_n$ in Inequality~\eqref{eq:dt_N_3}. From Theorem~\ref{th:propag_L}, we know that $L_n$ is bounded by a constant, hence by Inequality~\eqref{eq:bound_Mn_Ln}, we can bound $M_n$ by
		\begin{equation}\label{eq:bound_Mn_Ln_2}
			M_n \leq \frac{N}{t^n} +  \frac{C}{t^n} + C\,\sum_{k=1}^{n-1} t^{k-n} \,\tilde{Z}_k^n.
		\end{equation}
		To control $\tilde{Z}_k^n$, we use Inequality~\eqref{eq:Z_k_0} with $\tilde{\op}$. Together with the the induction hypothesis on $\tilde{Z}_{k-1}^{n-2}$, the fact that $\tilde{M}_n \leq C\,t^{-n}$ and the fact that $\tilde{N}_n = N_n \leq N$, we get
		\begin{equation}\label{eq:bound_tilde_Znk_N}
			\begin{aligned}
		 	\tilde{Z}_k^n &\leq C \,t^{-k} N^{1-\frac{k}{n}} + C\sum_{\ii=0}^{n} t^{2^{-\ii} - k} \weight{t}^{2^{-\ii}\(n-1-k\)}  N^{\(1-2^{-\ii}\)\(1-\frac{k}{n}\)}
		 	\\
		 	&\leq C\( t^{-k} N^{1-\frac{k}{n}} +  t^{1-k} \weight{t}^{n-1-k}  + t^{\eps - k} \weight{t}^{\eps\(n-1-k\)} N^{\(1-\eps\)\(1-\frac{k}{n}\)}\).
			\end{aligned}
		\end{equation}
		where the second inequality follows from Young's inequality for the product. Putting these inequalities back into Formula~\eqref{eq:bound_Mn_Ln_2} and using again Young's inequality, we obtain for any $t\geq \tau > 0$
		\begin{equation}\label{eq:bound_Mn_N}
			M_n \leq \frac{N}{t^n} + C_\tau \(\frac{N^{1-\frac{1}{n}}}{t^n} +  \frac{1}{t}\)
		\end{equation}
		for some constant $C_\tau$ depending on $\tau$. Together with \eqref{eq:dt_N_3} and the fact that $t\geq\tau$ and $N\geq 1$, this implies
		\begin{align}\nonumber
			\ddt{N} &\leq n\,\frac{N}{t} + C_\tau\( \frac{N^{1-\frac{1}{n^2}}}{t} + \frac{N^{1-\frac{1}{n}}}{t^\frac{1}{n}} + t^{n-2} + \frac{N^{1-\eps}}{t^{1-\eps\(n-1\)}} + \frac{N^{\(1-\eps\)\(1-\frac{1}{n}\)}}{t^{\frac{1}{n}-\eps\(n-2-\frac{1}{n}\)}}\).
		\end{align}
		or equivalently, for $\bar{N} = N\,t^{-n}$,
		\begin{align*}
			 \ddt{\bar{N}} &\leq C_\tau \(\frac{\bar{N}^{1-\frac{1}{n^2}}}{t^{1+\frac{1}{n}}} + \frac{\bar{N}^{1-\frac{1}{n}}}{t^{1+\frac{1}{n}}} + t^{-2} + \frac{\bar{N}^{1-\eps}}{t^{1+\eps}} + \frac{\bar{N}^{\(1-\eps\)\(1-\frac{1}{n}\)}}{t^{\(1+\eps\)\(1+\frac{1}{n}\)}}\) =: F(\bar{N},t)
		\end{align*}
		with $F\in C^1((0,\infty)\times[\tau,\infty))$. Taking $y(t)$ as the solution of $y'(t) = F(y(t),t)$ with initial condition $y(\tau) = \bar{N}(\tau)$, we see that $y'(t) \geq 0$, hence $y(t)\geq y(\tau)>0$. Therefore, up to multiplying by a constant depending on $\tau$, we can keep only the biggest powers of $y$ and $t$ in $F(y,t)$ if we just want to bound $y'$ by above. This leads to
		\begin{align*}
			 y'(t) &\leq C_\tau\, \frac{y(t)^{1-\eps}}{t^{1+\eps}}
		\end{align*}
		where we used the fact that $\eps^{-1} = 2^n \geq n^2 \geq n$ since $n\geq 4$. This implies that for any $t>\tau$, $y(t)^\eps \leq y(\tau)^\eps + C_\tau$ is bounded if $y(\tau)$ is bounded. However, we already know by Inequality~\eqref{eq:dt_N_2} that $N_n(t) \leq N_n^\init + C_T\,t^\frac{n}{2}$ for any $t\in[0,T]$ for some $T>0$. Therefore, taking for example $\tau = T$, we obtain that $y(\tau)= \(N_n(T)+1\)\,T^{-n}$ is bounded. Therefore, $y$ is bounded for $t\geq T$ and also $N_n$ since
		\begin{align*}
			N_n(t) &\leq N_n^\init + C_T\,t^\frac{n}{2} &\text{ if } t\leq T
			\\
			N_n(t) &= \bar{N}(t)\, t^n -1 \leq y(t)\,t^n \leq C_\tau\,t^n &\text{ if } t\geq T.
		\end{align*}
		The bound on $M_n$ is then an immediate consequence of Inequality~\eqref{eq:bound_Mn_N} for large times and the fact that $M_n$ is bounded on $[0,T]$, while the bound on $\tilde{Z}^n_k$ is a consequence of Formula~\eqref{eq:bound_tilde_Znk_N}.
	\end{demo}
	
	Actually, it is sufficient to use the condition of smallness of moments for $n=4$ to get a global propagation of higher moments as soon as $\fb_4 >\fb_n$ (which corresponds to $r>\frac{d}{d-1}$). This leads to the following proposition.
	\begin{prop}\label{prop:N4Mn}
		Under the condition of Theorem~\ref{th:regu_VP}, $M_n \in L^\infty_\loc(\R_+)$ and more precisely, there exists $c_n = c_{d,n,r} \geq 0$ and $C>0$ depending on the initial conditions such that
		\begin{equation*}
			M_n \leq C \weight{t}^{c_n}.
		\end{equation*}
	\end{prop}
	
	\begin{demo}[ of the proposition and of Theorem~\ref{th:regu_VP}]
		Since $\fb \geq \fb_4$ and $\fb\geq \frac{d}{3}$, we can use Proposition~\ref{prop:propag_MN} for $n=4$, and deduce
		\begin{equation*}
			M_4 \leq C,
		\end{equation*}
		for a given $C>0$. This already proves the result in the case $n=4$, so that we assume now that $n\geq 6$. Then, we use Formula~(44) from \cite{lafleche_propagation_2019}, which reads
		\begin{align}
			\ddt{M_n} &\leq C_{d,r,n}C_K \Nrm{\op^\init}{\L^r}^{\Theta_2} M_{n-2}^{\Theta_0}\, M_{n}^{\Theta},
		\end{align}
		with
		\begin{align*}
			\Theta &= 1 + \frac{n-1}{2}\(\frac{\fb_{n-2}}{\fb} - 1\).
			\\
			\Theta_0 &= (1-\eps)\(\frac{3}{2} - \frac{r'}{p'_{n-2}}\)
			\\
			\Theta_2 &= \frac{3}{2} - \Theta_1 - \Theta_0,
		\end{align*}
		where
		\begin{equation*}
			\eps = \frac{nr'+d}{(n-2)r'+3d}\(\frac{(n-2)r'+d}{\fb}-(n-2)\).
		\end{equation*}
		In particular, since $r\geq d'$, then $\fb_n$ is a non-increasing sequence and we deduce that for any $n\geq 6$, $\fb> \fb_4 \geq \fb_{n-2}$, which implies that $\Theta< 1$. We then obtain Inequality~\eqref{eq:asympt_Mn} by Gronwall's Lemma and by induction over $n\in 2\N$. From this bound, Formula~\eqref{eq:asympt_Nn} about $N_n$ can be deduced by using again Inequality~\eqref{eq:dt_N} and Gronwall's Lemma. Finally, since we know by Theorem~\ref{th:propag_L} that $L_4$ is bounded, the asymptotic behavior of $\rho$ in Formula~\eqref{eq:asympt_rho} is a consequence of Corollary~\ref{cor:LT_transporte}. The other inequalities follows from Proposition~\ref{prop:propag_MN}.
	\end{demo}
	
	\begin{demo}[ of Theorem~\ref{th:CV_VP}]
		The hypotheses of Theorem~\ref{th:regu_VP} are fulfilled with $r=\infty$, thus we deduce the existence of nonnegative constants $c$ and $\nu_M$ such that
		\begin{equation*}
			M_n \leq \nu_M \weight{t}^c.
		\end{equation*}
		Therefore, we can use \cite[Proposition~5.3]{lafleche_propagation_2019}, which tells us that for any $(n_0,n)\in (2\N)^2$ verifying $d<n_0 \leq \(1-\frac{1}{\fb}\)n + 1 - \frac{d}{\fb}$, it holds
		\begin{align}\nonumber
			c_{d,n_0}\Nrm{\rho(t)}{L^\infty} &\leq \Nrm{\op(t)}{\L^\infty(m)}
			\\\label{eq:propag_Linfty_m}
			&\leq 2^{n_0} \(\Nrm{\op^\init}{\L^\infty(m)} + C_{\op^\init} \(t+\int_0^tM_n^{1-\frac{1}{\fb}}\)^{n_0}\)
			\\\nonumber
			&\leq 2^{n_0} \(\Nrm{\op^\init}{\L^\infty(m)} + \,C_{\nu_M,\op^\init} \weight{t}^{n_0\(1+\frac{c}{\fb'}\)}\),
		\end{align}
		where $m = 1+\sum_{\ii=1}^d \opp_\ii^{n_0}$ and $C_{\op^\init} = \(C_{d,n,n_0} \Nrm{\nabla K}{L^\fb} (1+M_0)\)^{n_0} \Nrm{\op^\init}{\L^\infty}^{1+\frac{n_0}{\fb}}$.
		This proves \eqref{eq:rho_bound}. As in \cite[Section~4]{golse_schrodinger_2017}, we then define the time dependent coupling $\gam = \gam(t,z)$ with $z=(x,\xi)$ as the solution to the Cauchy problem 
		\begin{equation*}
			\partial_t\gam = \(-v\cdot\nabla_x-E\cdot\nabla_\xi\)\gam + \frac{1}{i\hbar}\com{H,\gam},
		\end{equation*}
		with initial condition $\gam^\init \in \C_\hbar(f^\init,\op^\init)$. As proved in \cite[Lemma 4.2]{golse_schrodinger_2017}, the coupling property is preserved by the dynamics, and so $\gam \in \mathcal{C}(f(t),\op(t))$. We also define with the notations of \eqref{def:Wh} the quantity
		\begin{equation*}
			\Eh = \Eh(t) := \intdd \Tr{\ch(z)\gam(z)}\d z,
		\end{equation*}
		so that by the definition~\eqref{def:Wh} of $\Wh$, we have
		\begin{align}\label{eq:Wh_Eh}
			\Wh(f(t),\op(t))^2 &\leq \Eh.
		\end{align}
		Moreover, remark that by \cite[Theorem~2.4]{golse_schrodinger_2017}, the left-hand side here is bigger or equal to $d\hbar$, so that $\Eh\geq d\hbar$. Then, as in~\cite[Proof of Proposition~6.3]{lafleche_propagation_2019}, we obtain
		\begin{align*}
			\frac{\d\Eh}{\d t} &\leq \Eh + \sqrt{2\Eh}\(C_1^2\(\Eh+d\hbar\) + C_2^2\(\Eh + \frac{1}{4}\Eh \ln(\Eh)^2\)\)^\frac{1}{2}
		\end{align*}
		with $C_1 = C_K \max\!\(\Nrm{\rho(t)}{L^\infty},\Nrm{\rho_f(t)}{L^\infty}\)^\frac{1}{2}\Nrm{\rho(t)}{L^\infty}^\frac{1}{2}$ and $C_2= C_K \Nrm{\rho(t)}{L^\infty}$ with $C_K = C_d \Nrm{\nabla K}{B^1_{1,\infty}}$. Using the facts that $\Eh\geq d\hbar$ and $C_2\leq C_1$, and dividing the inequality by $\Eh$, this leads to
		\begin{align*}
			\frac{\d\ln(\Eh)}{\d t} &\leq 1 + \(6\,C_1^2 + \frac{C_2^2}{2} \ln(\Eh)^2\)^\frac{1}{2} \leq 1 + \sqrt{6}\, C_1 + C_2 \n{\ln(\Eh)}.
		\end{align*}
		Remark that the right-hand side is a Lipschitz function of $\ln(\Eh)$. Denoting by $y(t,y_0)$ the solution of $\partial_t y = 1+\sqrt{6}\,C_1 + C_2 \n{y}$ with initial condition $y_0$, we deduce by Inequality~\eqref{eq:Wh_Eh} and comparison of ordinary differential equations that
		\begin{align*}
			\Wh(f(t),\op(t))^2 \leq \Eh(t) \leq e^{y(t,\ln(\Eh(0)))}.
		\end{align*}
		Since $y(t,\cdot)$ is a non-decreasing function, minimizing over $\Eh(0)$ yields
		\begin{align}\label{eq:Wh_vs_y}
			\Wh(f(t),\op(t))^2 \leq e^{y(t,2\ln(\Wh(f(t),\op(t))))}.
		\end{align}
		Recalling that $1+\sqrt{6}C_1$ and $C_2$ are bounded by above by a function of the form $b(t) = C^\init\bangle{t}^{n_0\(1+\frac{c}{\fb'}\)}$, we find by solving the equation $y'=b(t)\(1+\n{y}\)$ that
		\begin{align*}
			y(t) &\leq \(y(0) + 1\) e^{B(t)} - 1 &&\text{if } y(0)\geq 0
			\\
			y(t) &\leq y(0)\, e^{-B(t)} +  1 - e^{-B(t)} &&\text{if } y(0)\leq 0 \text{ and } y(t)\leq 0
			\\
			y(t) &\leq \frac{e^{B(t)}}{1-y(0)} -1 &&\text{if } y(0)\leq 0 \text{ and } y(t)\geq 0
		\end{align*}
		where $B(t) = \int_0^t b(s)\dd s$ and $C^\init = 1 + C_d\Nrm{\nabla K}{B^1_{1,\infty}} \sup_{t} \frac{\Nrm{\rho_f(t)}{L^\infty} + \Nrm{\rho(t)}{L^\infty}}{\weight{t}^{n_0(1+c/\fb')}}$ depend only on the initial conditions by Inequality~\eqref{eq:propag_Linfty_m}. Remark that in the last case, we have $0 \leq \(y(0)-1\)e^{-B(t)}+1$, thus we obtain that $y(t) \leq e^{B(t)} -1 \leq \(y(0)-1\)e^{-B(t)}+e^{B(t)}$. Therefore, we can summarize the inequalities for any values of $y(0)$ by
		\begin{align*}
			y(t) \leq y(0)\, e^{\sign(y(0)) B(t)} + e^{B(t)}.
		\end{align*}
		We conclude by combining this inequality with Formula~\eqref{eq:Wh_vs_y}.
	\end{demo}
	
\section{Acknowledgments}

	This work has been supported by Université Paris-Dauphine, PSL Research University.


\renewcommand{\bibname}{\centerline{Bibliography}}
\bibliographystyle{abbrv} 
\bibliography{Vlasov}

\end{document}